\documentclass[leqno,11pt]{amsart}
\usepackage[colorlinks,pagebackref,hypertexnames=false]{hyperref}

\usepackage{amsmath,amsthm,amssymb,mathrsfs} 
\usepackage[alphabetic,backrefs]{amsrefs}
\usepackage{ae,aecompl}
\usepackage[margin=1in]{geometry} 

\usepackage[english]{babel}

\usepackage{bookmark}

\numberwithin{equation}{section}

\DeclareMathOperator\Ad{Ad}

\newcommand{\abs}[1]{\left\lvert #1\right\rvert}
\newcommand{\norm}[1]{\left\lVert #1\right\rVert}
\newcommand{\ip}[2]{\left\langle #1,#2\right\rangle}
\newcommand{\vol}{\mathrm{vol}}

\newcommand{\dist}{\operatorname{dist}}

\newcommand{\eps}{\varepsilon}
\newcommand{\wto}{\rightharpoonup}

\newcommand{\SU}{\mathrm{SU}}
\newcommand{\SO}{\mathrm{SO}}

\newtheorem{theorem}{Theorem}

\newtheorem{corollary}{Corollary}
\newtheorem{lemma}{Lemma}
\newtheorem{proposition}{Proposition}
\theoremstyle{definition} \newtheorem{definition}{Definition}

\newtheorem{remark}{Remark}

\newtheorem*{question}{Question}
\newtheorem*{acknowledgements}{Acknowledgements}

\numberwithin{theorem}{section}
\numberwithin{proposition}{section}
\numberwithin{corollary}{section}
\numberwithin{lemma}{section}
\numberwithin{remark}{section}

\author{Daniel Fadel}
\address{Instituto de Ci\^encias Matem\'aticas e de Computa\c{c}\~ao,
Universidade de S\~ao Paulo, S\~ao Carlos--SP, Brazil}
\email{daniel.fadel@icmc.usp.br}
\urladdr{https://sites.google.com/view/daniel-fadel-math-homepage/home}

\author{Gon\c{c}alo Oliveira}
\address{Instituto Superior T\'ecnico, Universidade de Lisboa, Lisbon,
Portugal}
\email{galato97@gmail.com}
\urladdr{https://sites.google.com/view/goncalo-oliveira-math-webpage/home}

\title[The limit of large mass monopoles]{The limit of large mass monopoles} 

\date{}

\begin{document}

%===============================================================================
\begin{abstract}
In this paper, we consider finite energy $\SU(2)$ monopoles on an asymptotically conical, oriented Riemannian $3$-manifold with one end. The connected components of the moduli space of monopoles in this setting are labeled by an integer called the charge. We analyze sequences of monopoles with fixed charge and unbounded mass, or equivalently unbounded Yang--Mills--Higgs energy. We prove that their limiting behavior is characterized by energy concentration along a blow-up set, which is finite, and obtain effective bounds on its cardinality depending only on the charge. We also consider the zero set formed by the accumulation points of zeros of the Higgs fields and prove that the zero set and the blow-up set coincide.

At each concentration point, a moving centre analysis produces a finite cluster of mass one Euclidean monopoles whose total energy equals the concentration weight in the limiting energy measure, with no loss of mass-renormalized energy in the intervening regions. For general Yang--Mills--Higgs critical points whose energies are $O(m_i)$ as $i\to\infty$, where $m_i$ are the masses of the configurations, the blow-up set is countable and contains the limiting Higgs-zero set for every compact structure group. If the structure group is $\SU(2)$ or $\SO(3)$, a moving centre and scale selection argument also produces a nontrivial mass one Euclidean Yang--Mills--Higgs critical point at every blow-up point. We explain why this last conclusion can fail in higher rank.
\end{abstract}

\maketitle

\noindent\textbf{Changes to the published version and guide to this revision.}
The published version of this paper corresponds to arXiv version~4. The
present version~5 is a self-contained corrected and expanded version of that
article. It incorporates all corrections to the statements and proofs, as
well as the additional results on the complete bubble-cluster structure,
while retaining the organization, numbering, and unaffected parts of the
published article as far as possible.

The principal conclusions of the published paper remain valid. Theorem
\ref{thm: Main_Monopoles}\textup{(a)--(c)} is unchanged at the level of
its original conclusions, but the fixed centre argument is replaced by a
moving centre analysis. The substantive correction to the statement of
Theorem~\ref{thm: Main_Monopoles} is confined to part~\textup{(d)}: the displayed energy density must contain
a factor of $1/2$, and the coefficient of the limiting measure at a point
$x$ is $4\pi K_x$, where $K_x$ is the sum of the charges of all mass one
Euclidean monopoles in the finite cluster over $x$, rather than the charge
of one selected bubble. Theorem~\ref{thm: Main_YMH}\textup{(a')} remains
valid for the rank one groups $G=\SU(2)$ and $G=\SO(3)$ after a
moving centre and scale selection argument. For arbitrary compact $G$, the conclusions $Z\subset S$ and the countability of $S$
remain valid. Part~\textup{(b')} is restated using the uniform-threshold
definition of $S$ in \eqref{eq:main-concentration-zero-sets}, correcting
the stronger set-theoretic formula displayed in the published statement.
A higher rank example in Remark~\ref{rem:higher-rank-obstruction} explains
why the mass one conclusion in part~\textup{(a')} requires the rank one
restriction.

As an additional consequence of the corrected rank one bubbling argument,
Remark~\ref{rem:charged-bubbles-higgs-zeros} shows that whenever the
Euclidean bubble has nonzero magnetic charge, the moving centres may be
chosen among zeros of the original Higgs fields. In particular, the set
of concentration points admitting a charged bubble is contained in $Z$.

The necessity of the correction to
Theorem~\ref{thm: Main_Monopoles}\textup{(d)} is illustrated in
Section~\ref{sec:hierarchical-example}. There we construct sequences for
which several mass one Euclidean monopoles converge to the same point
$x\in S$ on the original manifold, while separating after rescaling by
the mass, possibly through several intermediate separation scales.
Consequently, the coefficient of the limiting measure at $x$ cannot in
general be recovered from any one pointed Euclidean limit: it is the
total energy, equivalently $4\pi$ times the total charge $K_x$, of the
complete cluster lying over $x$.

The supporting corrections retain the natural coercive terms in the
Bochner inequality, give a corrected $\varepsilon$-regularity proof
uniform in the mass, and replace the fixed centre steps in the bubbling
arguments by moving centres. Several aspects of this revised analysis,
in particular the rescaled $\varepsilon$-functional viewpoint, the
mass-uniform local estimates, and the moving centre bubble extraction,
are closely related to the methods developed in
\cite{cheng2025su2}. The arguments are included here in the form needed
for the present zero-potential problem. The additional results prove a
complete finite cluster decomposition for monopoles, identify every
concentration weight with the total energy and charge of its cluster,
establish the no loss identity, and give an explicit example with
several intermediate separation scales.

In addition, the monopole statements are now formulated under the
natural assumption of finite Yang--Mills--Higgs energy. More generally,
\cite{fadel2023asymptotics} proves that every finite energy configuration
on an AC $3$-manifold has a well-defined mass, and that the Higgs norm
of a finite energy YMH configuration converges uniformly to this mass
along the end. For $\SU(2)$ YMH critical points,
that work also proves a sharp asymptotic expansion of the Higgs field
norm; see Theorems~4.11 and~4.13 and Corollary~4.14 therein. Moreover,
a critical Sobolev estimate obtained from the same paper gives an
alternative proof of the small Higgs field radius bound for finite
energy configurations with arbitrary compact structure group. Thus the
stronger asymptotic condition on the connection used in the published
version is unnecessary throughout the results of the present paper.

For ease of comparison, material newly added or substantially revised in
version~5 is indicated at the beginnings of the relevant sections and in
remarks explicitly labeled as updates to version~5.

Accordingly, the present version may be read independently of the published version.

\medskip
%===============================================================================\medskip
%===============================================================================

%===============================================================================
\tableofcontents

\section{Introduction}\label{sec: intro}
Let $(X^3,g)$ be an oriented Riemannian $3$-manifold, and let $E$ be a
$G$-bundle over $X$, where $G$ is a compact Lie group with Lie algebra
$\mathfrak g$. Fix a positive-definite $\Ad$-invariant inner product on
$\mathfrak g$, and equip the associated adjoint bundle $\mathfrak g_E$
with the induced metric. All norms and inner products involving
$\mathfrak g_E$-valued forms are understood with respect to this metric
and the Riemannian metric $g$. When $G=\SU(2)$, we use throughout the
normalization
\[
    \langle a,b\rangle=-2\operatorname{tr}(ab),
    \qquad a,b\in\mathfrak{su}(2).
\]
We denote by $\mathcal{A}(E)$ the space of smooth connections on $E$
and refer to sections of the adjoint bundle $\mathfrak{g}_E$ as
\emph{Higgs fields}. A pair
$(A,\Phi)\in\mathcal{A}(E)\times\Gamma(\mathfrak{g}_E)$ is called a
\emph{configuration} on $E$. Such a configuration is called a
\emph{monopole} if it satisfies the Bogomolny equation
\begin{equation}\label{eq: monopole}
    \ast F_A=\text{d}_A\Phi.
\end{equation}
Its \emph{Yang--Mills--Higgs energy} is
\[
    \mathscr E_X(A,\Phi)
    :=
    \frac12
    \int_X
    \left(
        \lvert F_A\rvert^2
        +
        \lvert\text{d}_A\Phi\rvert^2
    \right)
    \vol_g.
\]
A configuration, and in particular a monopole, is said to have
\emph{finite energy} when $\mathscr E_X(A,\Phi)<\infty$.

The Bogomolny equation \eqref{eq: monopole} arises from dimensional
reduction of the instanton equations in four dimensions, and monopoles
form a special class of critical points of $\mathscr E_X$. Indeed, the
Euler--Lagrange equations of $\mathscr E_X$ are
\begin{equation}\label{eq: ymh}
    \text{d}_A^{\ast}F_A=[\text{d}_A\Phi,\Phi],
    \qquad
    \Delta_A\Phi=0,
\end{equation}
and, using the Bianchi identity $\text{d}_AF_A=0$, monopoles are easily
seen to satisfy them. We shall refer to solutions of \eqref{eq: ymh} as
\emph{Yang--Mills--Higgs (YMH) configurations}. A finite energy YMH
configuration will also be called a \emph{YMH critical point}.
Accordingly, whenever the finite energy condition is understood, we use
these terms interchangeably.

Any YMH configuration $(A,\Phi)$ satisfies $\Delta_A\Phi=0$ and hence
\begin{equation}\label{eq: subharmonic}
\Delta\frac{\lvert\Phi\rvert^2}{2} = \langle\Phi,\Delta_A\Phi\rangle - \lvert\text{d}_A\Phi\rvert^2 = - \lvert\text{d}_A\Phi\rvert^2\leqslant 0.
\end{equation}
As a consequence, the function $\lvert\Phi\rvert^2$ is subharmonic and therefore has no local maxima. In particular, if $X$ were compact and without boundary, then $\lvert\Phi\rvert^2$ would be constant and $\text{d}_A\Phi=0=\text{d}_A^*F_A$, in which case $A$ would be a Yang--Mills connection. Thus, if one is to study irreducible YMH configurations, meaning those with $\text{d}_A\Phi\neq0$, the manifold $X$ must be noncompact.\footnote{Other options would be to work on manifolds with nonempty boundary and/or to consider singular YMH configurations.}
YMH configurations, more specifically finite energy monopoles, have been the focus of intense study in conformally flat manifolds such as $\mathbb{R}^3$ (some of the earlier references in the mathematics literature are \cites{Taubes1982,Jaffe1980,Atiyah88}) and $\mathbb{R}^2\times S^1$ (see, for example, \cites{Cherkis2001,Cherkis2002,Foscolo2016}), as in these cases the corresponding moduli spaces are (noncompact) Hyperk\"ahler manifolds. In more general geometries, Braam \cite{Braam1989} considered monopoles on asymptotically hyperbolic manifolds, while Floer \cite{Floer1995} and Ernst \cite{Ernst1995} studied monopoles on asymptotically Euclidean (AE) ones, which are natural generalizations of the $\mathbb{R}^3$ situation. A further generalization of the $\mathbb{R}^3$ situation, containing the AE case as a subcase, is provided by asymptotically conical (AC) manifolds; see \cites{Kottke,oliveira2016monopoles}. These are complete Riemannian manifolds that, outside a compact set, are asymptotic to a metric cone over a closed $2$-dimensional Riemannian manifold, say $N^2$. We consider the case in which $N$ is connected; equivalently, the case in which $X$ has only one end.

Subsequent work by the first author
\cite{fadel2023asymptotics} showed that, on a one-ended AC
$3$-manifold, the relevant asymptotic quantities arise intrinsically
from finite Yang--Mills--Higgs energy.\footnote{This finite energy
formulation postdates the published version of the paper, where mass was
introduced through the stronger explicit asymptotic condition in
Definition~\ref{def: finite_mass}; see also
Remarks~\ref{rem:original-finite-mass-hypothesis} and
\ref{rem:finite-energy-monopole-asymptotics}.}
Accordingly, all global monopole results in this paper are formulated for finite energy monopoles.

Every finite energy configuration $(A,\Phi)$ has a unique canonically
associated number $m\geqslant0$, characterized by
\[
    m-\abs{\Phi}\in L^6(X).
\]
If $(A,\Phi)$ is a Yang--Mills--Higgs critical point, then the norm of the
Higgs field converges uniformly to $m$ along the end:
\begin{equation}\label{eq: mass}
    \lim_{\rho\to\infty}\abs{\Phi}=m.
\end{equation}
Moreover, $\abs{\Phi}\leqslant m$ on $X$; in particular, if $m=0$, then
$\Phi\equiv0$. We call $m$ the \emph{mass} of the configuration.

Suppose now that $(A,\Phi)$ is such a critical point, that $m>0$, and
that $G=\SU(2)$. The \emph{charge}, or \emph{monopole number}, of
$(A,\Phi)$ is defined by
\begin{equation}\label{eq: charge_integer}
    k=k(A,\Phi)
    :=
    \frac{1}{4\pi m}
    \int_X\langle F_A\wedge\text{d}_A\Phi\rangle.
\end{equation}
It is an integer and may equivalently be computed at infinity. Indeed,
there exists $R_0\gg_{A,\Phi}1$ such that $\Phi$ is nowhere vanishing on
$\{\rho\geqslant R_0\}$ and, writing
$N_R:=\{\rho=R\}\cong N$, one has
\begin{equation}\label{eq: charge_calculation}
    k
    =
    \lim_{R\to\infty}
    \frac{1}{4\pi}
    \int_{N_R}
    \abs{\Phi}^{-1}\langle\Phi,F_A\rangle.
\end{equation}

For every sufficiently large $R$, the normalized Higgs field determines a map
\[
    \frac{\Phi}{\abs{\Phi}}
    \colon
    N_R\cong N
    \longrightarrow
    \mathbb S^2\subset\mathfrak{su}(2),
\]
whose homotopy class is independent of the sufficiently large level
$R$, and whose Brouwer degree is equal to $k$. Alternatively, the
eigenspaces of $\Phi$ determine a splitting
\[
    E|_{N_R}\cong L\oplus L^{-1},
\]
where $L\to{N_R}$ is the eigenline bundle corresponding to the
preceding degree convention. Its degree is independent of $R$ and
satisfies
\[
    \deg(L)=k.
\]

If $(A,\Phi)$ is in addition an irreducible monopole, equivalently
$\text{d}_A\Phi\not\equiv0$, then $m>0$ and its charge is necessarily
positive,
\[
    k\in\mathbb Z_{>0},
\]
the Higgs field must vanish somewhere,
\[
    \Phi^{-1}(0)\neq\varnothing,
\]
and
\begin{equation}\label{eq: monopole_energy}
    \mathscr E_X(A,\Phi)=4\pi mk.
\end{equation}

More generally, for any finite energy $\SU(2)$ configuration of positive
mass $m$ and positive charge $k$, the Bogomolny completion gives
\begin{equation}\label{eq: Energy_Formula}
    \mathscr E_X(A,\Phi)
    =
    4\pi mk
    +
    \frac12
    \norm{\ast F_A-\text{d}_A\Phi}_{L^2(X)}^2.
\end{equation}
Thus, among finite energy configurations with fixed positive mass and positive charge, the finite energy monopoles are precisely the absolute minimizers of the Yang--Mills--Higgs energy.

The assertions above follow from
\cite{fadel2023asymptotics}*{Theorems~1.1 and~1.4 and
Corollary~1.9}; see also
\cite{fadel2023asymptotics}*{Remarks~1.8 and~1.10}. The virtual dimension of the moduli space of finite energy monopoles on an AC manifold was computed in \cite{Kottke}, and a smooth open set was constructed by a gluing theorem in \cite{oliveira2016monopoles}. This is an AC version of Taubes' original gluing construction for well-separated multimonopoles on $\mathbb{R}^3$; see \cite{Jaffe1980}. In the case of \cite{oliveira2016monopoles}, the mass plays the role of a parameter controlling the concentration of the resulting multi-monopole around its centers. Indeed, allowing the mass to vary gives the freedom of bringing these centers as close as one wants. In order to motivate the main results of this paper we shall now summarize this construction of large mass, charge $k$ monopoles on $X$. This goes as follows: Start with $k$ points in $X$; Insert charge one and mass one monopoles in $\mathbb{R}^3$ scaled down to fit in small disjoint balls around these points; As a byproduct of having been scaled down the monopoles must have mass larger than $O(d^{-2})$, where $d$ is the minimum separation between the $k$-points; Then, by making use of a partition of unity these can be glued with a certain mass $O(d^{-2})$ monopole in the complement of these balls; The resulting configuration does not solve the monopole equations, but by a version of the contraction mapping principle it can be deformed to a nearby one which does. Moreover, we further remark that this configuration produces monopoles with any mass $m \geq O(d^{-2})$, for more details and the precise statements see Theorem 1 in \cite{oliveira2016monopoles} or Theorem \ref{thm: Monopoles_Examples} later in this paper. \footnote{We further point out that it should be possible to start this construction by using higher charge monopoles in $\mathbb{R}^3$ (monopole clusters). A metric version of this gluing has been carried out in \cite{KottkeSinger} for the case of $\mathbb{R}^3$.}

The goal of this paper is to take the inverse point of view and consider a sequence of finite energy monopoles $\lbrace (A_i , \Phi_i) \rbrace_{i \in \mathbb{N} }$ with unbounded masses, $\limsup m_i = \infty$, but fixed charge $k$, over an AC manifold $(X^3,g)$. In this case, the natural expectation would be an inverse construction to that of \cite{oliveira2016monopoles}, with the monopoles either ``escaping'' through the end, or getting concentrated around at most $k$ points $x_1, \ldots , x_k$ in $X$, where a monopole in the Euclidean $\mathbb{R}^3 \cong T_{x_i} X$ bubbles off.\footnote{For general sequences of Yang--Mills--Higgs critical points whose
energies grow at most linearly with their masses,
Theorem~\ref{thm: Main_YMH} gives the inclusion $Z\subset S$ and the countability of $S$. If $G=\SU(2)$ or $G=\SO(3)$, it also gives a nontrivial mass one Euclidean Yang--Mills--Higgs bubble at every point of $S$; see Section~\ref{sec:general-YMH-bubbling}.} See Section \ref{sec: examples} in this paper for a plethora of examples motivating this expectation.

When the energies $\mathscr E_X(A_i,\Phi_i)$ remain uniformly bounded, the limiting behavior is well understood. The monopoles may escape through the end, but, after passing to a subsequence and changing gauge, their restrictions to every compact subset of $X$ converge smoothly to a monopole; see, for example, \cite{Atiyah88} for the more general statement on $\mathbb R^3$. Therefore, the most interesting case is when these energies do not remain bounded. Indeed, the energy formula \eqref{eq: Energy_Formula} for monopoles $\mathscr{E}_X(A_i , \Phi_i)=4 \pi k m_i$, shows that this is precisely the case under consideration, where the sequence of masses $m_i$ is unbounded.

We now introduce some preparation needed in order to state our main results. 
Let $\lbrace(A_i,\Phi_i)\rbrace_{i\in\mathbb N}
\subseteq\mathcal A(E)\times\Gamma(\mathfrak g_E)$ be a sequence of
Yang--Mills--Higgs critical points on $(X^3,g)$ whose masses satisfy
$\limsup_i m_i=\infty$. Define the \emph{blow-up set} $S$ and the \emph{zero set} $Z$ by
\begin{equation}\label{eq:main-concentration-zero-sets}
\begin{split}
S
&:=
\bigcup_{\varepsilon>0}
\bigcap_{0<r\leqslant\rho_0}
\left\{
x\in X:
\liminf_{i\to\infty}
m_i^{-1}\mathscr{E}_{B_r(x)}(A_i,\Phi_i)
\geq\varepsilon
\right\},
\\
Z
&:=
\bigcap_{n\geq1}
\overline{\bigcup_{i\geq n}\Phi_i^{-1}(0)}.
\end{split}
\end{equation}
The set $S$ records the points where the mass-renormalized energy
concentrates, while $Z$ consists of the accumulation points of the
Higgs-field zeros. For finite energy $\SU(2)$ monopoles of fixed charge,
Theorem~\ref{thm: Main_Monopoles} shows that these sets coincide and that
each concentration point carries a complete finite cluster of mass one
Euclidean monopoles. For general sequences of Yang--Mills--Higgs critical points whose
energies grow at most linearly with their masses,
Theorem~\ref{thm: Main_YMH} gives the
inclusion $Z\subset S$ and the countability of $S$ for arbitrary compact
$G$, and restores the mass one bubbling conclusion when
$G=\SU(2)$ or $G=\SO(3)$. Here $\mathcal H^0$ denotes the counting measure on $X$.

\begin{theorem}\label{thm: Main_Monopoles}
Let $(X^3,g)$ be an AC oriented Riemannian $3$-manifold with one end,
let $E$ be an $\SU(2)$-bundle over $X$, and let
$\{(A_i,\Phi_i)\}_{i\in\mathbb N}$ be finite energy monopoles of fixed
charge $k>0$ whose masses satisfy $\limsup_i m_i=\infty$. Then there is a
subsequence, still denoted by the same indices, for which $m_i\to\infty$
and
\[
    \mu_i
    :=
    \frac{1}{2m_i}
    \left(
       \lvert F_{A_i}\rvert^2+\lvert\text{d}_{A_i}\Phi_i\rvert^2
    \right)\vol_g
    \rightharpoonup\mu.
\]
With $S$ and $Z$ defined from this subsequence as above, the following
hold.
\begin{enumerate}
\item[(a)] For each $x\in S$, there exist $N_x\geqslant1$, sequences of
zeros $p_{i,x,\beta}\in\Phi_i^{-1}(0)$ converging to $x$, and mass one
Euclidean monopoles $(A_{x,\beta},\Phi_{x,\beta})$ of positive charges
$q_{x,\beta}$, $1\leqslant\beta\leqslant N_x$, obtained by rescaling
about these centres. Distinct centres separate at the mass scale:
\[
    m_i\,d(p_{i,x,\beta},p_{i,x,\gamma})\longrightarrow\infty
    \qquad(\beta\neq\gamma).
\]
If
\[
    K_x:=\sum_{\beta=1}^{N_x}q_{x,\beta},
\]
then
\[
    \mu(\{x\})
    =4\pi K_x
    =\sum_{\beta=1}^{N_x}
      \mathscr E_{\mathbb R^3}(A_{x,\beta},\Phi_{x,\beta}).
\]
In particular, $1\leqslant N_x\leqslant K_x$. More precisely, if
$r>0$ is chosen so that
$\overline{B_{2r}(x)}\cap S=\{x\}$ and
$\mu(\partial B_r(x))=0$, then
\[
 \lim_{R\to\infty}\limsup_{i\to\infty}
 \mu_i\left(
 B_r(x)\setminus
 \bigcup_{\beta=1}^{N_x}B_{Rm_i^{-1}}(p_{i,x,\beta})
 \right)=0.
\]
Choosing any one $\beta$ gives the mass one monopole asserted in the
published part~\textup{(a)}, with
$q_{x,\beta}\leqslant K_x\leqslant k$.

\item[(b)] The blow-up set has the quantized description
\[
    S
    =
    \bigcap_{0<r\leqslant\rho_0}
    \left\{
       x\in X:
       \liminf_{i\to\infty}\mu_i(B_r(x))\geqslant4\pi
    \right\}
    =Z.
\]

\item[(c)] The common set $S=Z$ is finite and, when $S\neq\varnothing$,
\[
    \mathcal H^0(S)
    \leqslant
    \frac{k}{\min_{x\in S}K_x}
    \leqslant k.
\]

\item[(d)] The measures converge weakly as Radon measures according to
\[
    \mu_i
    \rightharpoonup
    4\pi\sum_{a=1}^{\ell}K_{x_a}\delta_{x_a},
    \qquad
    \sum_{a=1}^{\ell}K_{x_a}\leqslant k,
\]
where $S=\{x_1,\ldots,x_\ell\}$. Moreover,
\[
    \sum_{x\in S}N_x
    \leqslant
    \sum_{x\in S}K_x
    \leqslant k.
\]
The family $(\mu_i)$ is tight if and only if
$\sum_{x\in S}K_x=k$. In particular, the limiting measure has support
$S$.
\end{enumerate}
\end{theorem}

For a sequence of YMH critical points on a $G$-bundle,
the linear bound $\mathscr E_X(A_i,\Phi_i)=O(m_i)$ still implies that
the blow-up set is countable and contains the limiting Higgs-zero set. In
the monopole case, this energy bound follows from fixing the charge. The
mass one bubbling conclusion also remains valid for the rank one groups
$G=\SU(2)$ and $G=\SO(3)$; in higher rank an additional regularity
hypothesis on the mass-normalized asymptotic Higgs field is needed, as
explained in Remark~\ref{rem:higher-rank-obstruction}.

\begin{theorem}\label{thm: Main_YMH}
Let $(X^3,g)$ be an AC oriented Riemannian $3$-manifold with one end,
let $E$ be a $G$-bundle over $X$, where $G$ is compact, and let
$\{(A_i,\Phi_i)\}_{i\in\mathbb N}$ be Yang--Mills--Higgs critical
points whose masses satisfy $\limsup_i m_i=\infty$. Suppose that
\[
    m_i^{-1}\mathscr E_X(A_i,\Phi_i)\leqslant C.
\]
After passing to a subsequence for which $m_i\to\infty$, and defining
$S$ and $Z$ from this subsequence as above, the following hold.
\begin{itemize}
\item[(a')] If $G=\SU(2)$ or $G=\SO(3)$, then for every $x\in S$
there are centres $p_i\to x$ and scales $t_i\downarrow0$ satisfying
\[
    0<c_x\leqslant m_it_i\leqslant C_x<\infty.
\]
After passing to a further subsequence, the limit
\[
    \alpha
    :=
    \lim_{i\to\infty}\frac{m_i^{-1}}{t_i}
    \in(0,\infty)
\]
exists. If
\[
    s_i(z):=\exp_{p_i}(t_i z),
\]
then, after changing gauge, the pullbacks
\[
    s_i^*(A_i,m_i^{-1}\Phi_i)
\]
converge smoothly on compact subsets of $\mathbb R^3$ to a nontrivial
critical point $(A_x,\phi_x)$ of the scaled functional
$\mathcal Y_\alpha$ defined in
\eqref{eq:scaled-general-YMH-functional}. The pair
$(A_x,\alpha^{-1}\phi_x)$ is a Euclidean YMH critical point of mass
$\alpha^{-1}$, and the standard Euclidean Yang--Mills--Higgs scaling
then yields a nontrivial mass one Euclidean YMH critical point.

\item[(b')] For every compact $G$, the blow-up set contains the zero set:
\[
Z\subset S.
\]

\item[(c')] For every compact $G$, the sets
\[
S_j:=
\bigcap_{0<r\leqslant\rho_0}
\left\{
 x\in X:
 \liminf_{i\to\infty}
 m_i^{-1}\mathscr E_{B_r(x)}(A_i,\Phi_i)\geqslant j^{-1}
\right\}
\]
satisfy $S=\bigcup_{j\geqslant1}S_j$ and
$\mathcal H^0(S_j)\leqslant jC$. In particular, $S$, and hence $Z$, is
countable.
\end{itemize}
\end{theorem}

\begin{remark}[Rank one restriction in version~5]
\label{rem:rank-one-general-ymh-bubbling}
Part~\textup{(a')} of the published theorem was stated for an arbitrary compact semisimple structure group. The moving centre and scale
selection argument below proves it for $G=\SU(2)$ or $G=\SO(3)$. In
higher rank, the scale of a nonabelian component need not be determined
by the norm of the full Higgs field;
Remark~\ref{rem:higher-rank-obstruction} gives an explicit example. The
rank one restriction does not affect the inclusion $Z\subset S$ or the
countability conclusion, which hold for every compact structure group
equipped with the fixed Ad-invariant inner product chosen above. The
displayed description in part~\textup{(b')} is, however, written here
with the uniform positive threshold appearing in the definition of
$S$.
\end{remark}

The proof of these results follows from putting together a number of other results. In order to guide the reader on how these are put together, we shall now explain how this paper is organized. Section \ref{sec: preliminaries} collects background definitions, such
as the notion of AC manifolds and the original finite mass condition
used in the published version, together with a few technical results
needed later. We retain that condition in order to reproduce the
original asymptotic argument. However, as explained in
Remark~\ref{rem:original-finite-mass-hypothesis}, the finite energy
results of \cite{fadel2023asymptotics}, combined with a critical
Sobolev estimate, make this additional hypothesis unnecessary for all subsequent results
and for every compact structure group. Section \ref{sec: examples} gives several examples of families of monopoles whose masses converge to infinity. The results are very illustrative and allow for the realization of all cases in our Theorem \ref{thm: Main_Monopoles}, providing a good intuition for the behavior of large mass monopoles.

The proof of Theorems \ref{thm: Main_Monopoles}--\ref{thm: Main_YMH} takes up every section from Section \ref{sec: taubes} to \ref{sec: measures}, and their content is summarized below.

Having in mind the aim of relating the zero set $Z$ and the blow-up set ${S}$ of such sequences of large mass monopoles, Section \ref{sec: taubes} gives an AC version of Taubes' small Higgs field radius estimate (Theorem \ref{thm: taubes_estimate}). This provides a way to control how big, in terms of the mass $m\neq 0$ and the charge $k\neq 0$ of the monopole, one needs the radius of a ball in $X$ to be so that the value of the Higgs field outside such a ball is a sufficiently large portion of $m$. 

In Section \ref{sec: e-regularity} we prove an appropriate $\varepsilon$-regularity theorem (Theorem \ref{thm: e-regularity}) for Yang--Mills--Higgs configurations of mass $m\neq 0$, which is an important ingredient to relate the zero set with the blow-up set. Indeed, in Section \ref{sec: lowerbound}, using a simple argument involving the fundamental theorem of calculus, together with the $\varepsilon$-regularity, we prove that a large mass Yang--Mills--Higgs configuration with (locally) small energy has an interior lower bound on its Higgs field, provided it is bounded from below in some boundary ball. Together with our analogue of Taubes small Higgs field estimate, this is used in Section \ref{sec: blow_up_and_zero} to prove the inclusion $Z\subseteq{S}$, i.e. the last part of (b') in Theorem \ref{thm: Main_YMH}; here we also prove (c').

Section \ref{sec: bubbling} first repairs the moving centre and scale
selection argument for rank one Yang--Mills--Higgs critical points,
proving part~\textup{(a')} of Theorem~\ref{thm: Main_YMH} for
$G=\SU(2)$ or $G=\SO(3)$. It then uses the additional structure of
the Bogomolny equation to extract, over each $x\in S$, a complete finite
cluster of mass one Euclidean monopoles. Their total energy is the
concentration weight $4\pi K_x$, and no mass-renormalized energy remains between the profiles and the ambient scale. The
energy formula and a degree argument give $S=Z$.

Section \ref{sec: measures} records the resulting measure convergence and
the cardinality estimate in Theorem~\ref{thm: Main_Monopoles}. Parts
\textup{(b')} and~\textup{(c')} of Theorem~\ref{thm: Main_YMH} are proved
in Section~\ref{sec: blow_up_and_zero}.

\begin{acknowledgements}
The acknowledgements and funding information in this paragraph refer to
the original work leading to the published version of the paper. The
authors would like to thank Henrique Bursztyn and Vinicius Ramos for
their warm hospitality and for making possible the research visits of
the first author to IMPA, where a major part of the original work was
carried out. We also thank \'Akos Nagy, Thomas Walpuski, and the anonymous
referee of the published paper for their mathematical comments on an
earlier version, which allowed us to correct an error in
Theorem~\ref{thm: e-regularity} and a misstatement in
Proposition~\ref{prop: singular_harmonic}.

The authors are grateful to Fundação Serrapilheira, which supported
the first author during his visits to IMPA through a research grant held
by Vinicius Ramos. The first author is also grateful to CNPq for the
doctoral fellowship that supported his Ph.D.; the original article formed
part of his Ph.D. thesis at the University of Campinas. The second author
was supported by Fundação Serrapilheira, IMPA/CAPES, and CNPq.

At the time the original work was carried out, the first author was
affiliated with IMECC--Unicamp and the second author with Universidade
Federal Fluminense. The author affiliations displayed in the present
version are their current affiliations.
\end{acknowledgements}

\section{Preliminaries}\label{sec: preliminaries}
In this short section we collect a few background facts which will prove useful in the body of the paper. The reader familiar with the notion of asymptotically conical manifolds is welcome to skip this section and refer back to it as needed.

We start with a basic scaling property of the Yang--Mills--Higgs/monopole equations.
\begin{proposition}\label{prop: scaling}
	Let $(X^3,g)$ be an oriented Riemannian $3$-manifold, $E$ a $G$-bundle over $X$, and $(A,\Phi)$ a configuration on $E$.
	 If $(A,\Phi)$ is Yang--Mills--Higgs (resp. a monopole) on $(X^3,g)$, then $(A,\lambda^{-1}\Phi)$ is Yang--Mills--Higgs (resp. a monopole) on $(X^3,g_{\lambda}:=\lambda^2 g)$, for any $\lambda>0$.
\end{proposition}
\begin{proof}
	Acting on $k$-forms, the Hodge-$\ast$ operators associated to $g_{\lambda}$ and $g$ are related by $\ast_{\lambda} = \lambda^{3-2k} \ast$. Therefore, we have $\text{d}_A^{\ast_{\lambda}} F_A = \lambda^{-2} \text{d}_A^{\ast} F_A = \lambda^{-2} [\text{d}_A \Phi , \Phi]$ in the Yang--Mills--Higgs case, and $\ast_{\lambda} F_A = \lambda^{-1}\ast F_A = \lambda^{-1}\text{d}_A\Phi$ in the monopole case. The result follows.
\end{proof}
We shall work with the following class of noncompact Riemannian $3$-manifolds.
\begin{definition}\label{def:AC}
	Let $(X^3,g)$ be a complete oriented Riemannian $3$-manifold. Then $(X^3,g)$ is called \textbf{asymptotically conical} (AC) with rate $\nu<0$ if there exist a compact set $K\subset X$, an oriented, closed (compact and without boundary) Riemannian surface $(N^2,g_{N})$, and an orientation preserving diffeomorphism
	\[
	\varphi:C(N):=(1,\infty)_r\times N\to X\setminus K
	\] such that the cone metric $g_C:= dr^2+r^2g_{N}$ on $C(N)$ satisfies
	\[
	\lvert\nabla^j\left(\varphi^{\ast}g - g_C\right)\rvert_C = O(r^{\nu - j}),\quad\forall j\in\mathbb{N}_0.
	\] Here $\nabla$ is the Levi--Civita connection of $g_C$. We say furthermore that $X$ has \emph{one end} if $N$ is connected, and we refer to $X\setminus K$ as the \emph{end} of $X$. A \emph{distance function} on $X$ will be any positive smooth function $\rho:X\to\mathbb{R}^{+}$ such that $\rho|_{X\setminus K} = r\circ\varphi^{-1}$.
\end{definition}
On such manifolds, we shall be interested in the following particular class of configurations.
\begin{definition}\label{def: finite_mass}
	Let $(X^3,g)$ be an AC oriented Riemannian $3$-manifold with one end, as in Definition \ref{def:AC}, and let $E$ be a $G$-bundle over $X$. A configuration $(A,\Phi)$ on $E$ is said to have \textbf{finite mass} if the following holds. There exists a $G$-bundle $E_{\infty}$ over $N$ together with an isomorphism of principal bundles $\varphi^{\ast}\left(E|_{X\setminus K}\right)\cong\pi^{\ast}E_{\infty}$, where $\pi:(1,\infty)\times N\to N$ is the projection onto the second factor, and there exists a connection $A_{\infty}$ on $E_{\infty}$ such that $A$ is asymptotic to $A_{\infty}$ on $E_{\infty}$, i.e.
	\[
	\varphi^{\ast}\nabla_A = \pi^{\ast}\nabla_{\infty} + a,\quad\text{where }\quad\lvert\nabla_{\infty}^j a\rvert = O(\rho^{-1-j-\eta}),\quad\forall j\in\mathbb{N}_0\text{ and for some }\eta>0,
	\] and there is $m\in\mathbb{R}^{+}$ with
	\[
	\lim_{\rho\to\infty}\lvert\Phi\rvert = m.
	\] We call the constant $m$ the \textbf{mass} of $(A,\Phi)$.
\end{definition}

\begin{remark}[Finite energy update in version~5]
\label{rem:original-finite-mass-hypothesis}
Definition~\ref{def: finite_mass} is retained from the published
version in order to reproduce the original proof of
Proposition~\ref{prop:finite-mass-asymptotic-expansion} below. Its
asymptotic condition on the connection originates in
\cite{oliveira2014thesis}*{Definition~1.4.1}.

The later work \cite{fadel2023asymptotics} shows that this additional
condition is not needed for finite energy configurations. More precisely, for every compact structure group, an arbitrary finite
energy configuration has a unique number $m\in[0,\infty)$ such that
$m-\abs{\Phi}\in L^6(X)$. If it moreover satisfies
$\Delta_A\Phi=0$, then $\abs{\Phi}$ converges uniformly to $m$ along
the end; see
\cite{fadel2023asymptotics}*{Theorem~1.1 and Corollary~4.5}.
Furthermore, the proof of
\cite{fadel2023asymptotics}*{Theorem~1.1}, based on
Lemma~2.27 therein and the $L^2$-Sobolev inequality, yields
\begin{equation}\label{ineq: crit_Sobolev}
    \norm{m-\abs{\Phi}}_{L^6(X)}
    \lesssim
    \norm{\textnormal{d}\abs{\Phi}}_{L^2(X)}
    \leqslant
    \norm{\textnormal{d}_A\Phi}_{L^2(X)}.
\end{equation}
Here the last inequality is Kato's inequality. The estimate
\eqref{ineq: crit_Sobolev} is valid for every finite energy
configuration and every compact structure group.

For $G=\SU(2)$, the same paper proves the sharper asymptotic expansion
\begin{equation}\label{eq:finite-energy-asymp-behavior-YMH}
    \abs{\Phi}
    =
    m
    -
    \frac{\|\textnormal{d}_A\Phi\|_{L^2(X)}^2}
         {m\vol(N)}
    \frac1\rho
    +
    O(\rho^{-1-\mu})
    \qquad\text{as }\rho\to\infty
\end{equation}
for some $\mu>0$, whenever $(A,\Phi)$ is a YMH critical point of positive mass; see
\cite{fadel2023asymptotics}*{Theorems~4.11 and~4.13 and
Corollary~4.14}, especially equation~(4.27). This is stronger than \eqref{eq: asymp_behavior_YMH} and requires only finite energy.

Thus Definition~\ref{def: finite_mass} and
Proposition~\ref{prop:finite-mass-asymptotic-expansion} are retained
only to reproduce the original formulation and proof. All results below
that concern finite energy YMH configurations or YMH critical points
remain valid without the additional asymptotic-connection condition.
\end{remark}

\begin{remark}\label{rem: subharmonic}
	If $(A, \Phi)$ is a Yang--Mills--Higgs configuration, then $\Delta_A \Phi=0$ and thus $\vert \Phi \vert$ is subharmonic, as shown in \eqref{eq: subharmonic}. So if $(A,\Phi)$ is a Yang--Mills--Higgs configuration with finite mass $m$, then as $\vert \Phi \vert$ converges to $m$ along the end of $X$, the maximum principle yields that either $\vert \Phi \vert < m$ on $X$ or $\lvert\Phi\rvert$ is constant equal to $m$.   
\end{remark}
The next proposition, on the asymptotic behavior of the Higgs field norm of a finite mass Yang--Mills--Higgs configuration (cf. \cite{oliveira2014thesis}*{Section~1.4.1}), will be useful later in the proof of Theorem \ref{thm: taubes_estimate}.
\begin{proposition}\label{prop:finite-mass-asymptotic-expansion}
	Let $(X^3,g)$ be an AC oriented Riemannian $3$-manifold with one end, and let $E$ be a $G$-bundle over $X$. If $(A,\Phi)\in\mathcal{A}(E)\times\Gamma(\mathfrak{g}_E)$ is a finite mass Yang--Mills--Higgs configuration of mass $m\neq0$ such that $\lvert \textnormal{d}_A\Phi\rvert\in L^2(X)$, then, in a neighborhood $V(N)$ of the end,
	\begin{equation}\label{eq: asymp_behavior_YMH}
	\lvert\Phi\rvert = m - \frac{1}{\vol(N)}\frac{\|\textnormal{d}_A\Phi\|_{L^2(X)}^2}{m\rho}+o(\rho^{-1}).
	\end{equation}
	In particular, if $G=\SU(2)$ and $(A,\Phi)$ is a finite mass $m$ and charge $k$ monopole on $(X^3,g)$, then
	\begin{equation}\label{eq: asymp_behavior}
	\lvert\Phi\rvert = m - \frac{4\pi}{\vol(N)}\frac{k}{\rho}+o(\rho^{-1}).
	\end{equation}
\end{proposition}
\begin{proof}[Sketch of proof]
	By the arguments in \cite{oliveira2014thesis}*{proof of Proposition 1.4} (also see \cite{Jaffe1980}*{Theorem 10.5 in Chap. IV}), one can write
	\[
	\lvert\Phi\rvert = m - c\rho^{-1}+o(\rho^{-1})\quad\text{on $V(N)$},
	\] for some $c\in\mathbb{R}$ that we will now compute. Since $\lvert\text{d}_A\Phi\rvert\in L^2(X)$, by dominated convergence we can write
	\[
	\int_X \lvert\text{d}_A\Phi\rvert^2 = \lim_{r\to\infty} \int_{\rho^{-1}(0,r]}\lvert\text{d}_A\Phi\rvert^2.
	\] Now, since $\Delta_A\Phi=0$, we know that $\Delta\lvert\Phi\rvert^2 = -2\lvert\text{d}_A\Phi\rvert^2$. Hence, by Stokes' theorem,
	\[
	\int_{\rho^{-1}(0,r]}\lvert\text{d}_A\Phi\rvert^2 = \frac{1}{2}\int_{\rho^{-1}(r)} \ast\text{d}\lvert\Phi\rvert^2.
	\] Therefore, we can compute:
	\begin{eqnarray*}
		\int_X \lvert\text{d}_A\Phi\rvert^2 &=& \lim_{r\to\infty}\int_{\rho^{-1}(r)} \ast(\lvert\Phi\rvert\text{d}\lvert\Phi\rvert)\\
		&=& \lim_{r\to\infty}\int_{\rho^{-1}(r)}\lvert\Phi\rvert\partial_{\rho}\lvert\Phi\rvert\ast\text{d}\rho\\
		&=& \lim_{r\to\infty}\int_{\rho^{-1}(r)}\lvert\Phi\rvert\partial_{\rho}(m-c\rho^{-1}+o(\rho^{-1}))\rho^{2}\\
		&=& \lim_{r\to\infty}\int_{\rho^{-1}(r)} \lvert\Phi\rvert(c+o(1))\\
		&=& cm\vol(N),
	\end{eqnarray*} where in the last equality we used that $(A,\Phi)$ has finite mass equal to $m$. This proves equation \eqref{eq: asymp_behavior_YMH}, which in turn, in the case of monopoles, implies equation \eqref{eq: asymp_behavior} via the energy formula \eqref{eq: Energy_Formula}.\qedhere
\end{proof}

\begin{remark}[Finite energy asymptotics in version~5]
\label{rem:finite-energy-monopole-asymptotics}
For $\SU(2)$ YMH critical points, the preceding proposition is
superseded by the sharper finite energy expansion
\eqref{eq:finite-energy-asymp-behavior-YMH}. In particular, if
$(A,\Phi)$ is a finite energy $\SU(2)$ monopole of mass $m>0$ and
charge $k$, then the energy identity gives
\[
    \abs{\Phi}
    =
    m-\frac{4\pi k}{\vol(N)}\frac1\rho
    +O(\rho^{-1-\mu})
    \qquad\text{as }\rho\to\infty
\]
for some $\mu>0$; see
\cite{fadel2023asymptotics}*{Corollary~4.14}.

For an arbitrary compact structure group, the expansion
\eqref{eq:finite-energy-asymp-behavior-YMH} is not needed in the
arguments below. The critical Sobolev estimate
\eqref{ineq: crit_Sobolev} gives directly the alternative finite energy
version of the small Higgs field radius estimate recorded after
Theorem~\ref{thm: taubes_estimate}.
\end{remark}

Finally, we recall an auxiliary result from \cite{oliveira2016monopoles} for later reference. In what follows, for a point $x\in X$, we let $\delta_x\in\left(C_c^{\infty}(X)\right)'$ denote the Dirac delta distribution supported at $x$, and we consider the transpose of the Laplace operator, still denoted by $\Delta$, acting on $\left(C_c^{\infty}(X)\right)'$ in the usual fashion.
\begin{proposition}[{\cite{oliveira2016monopoles}*{Proposition~2}}]
\label{prop: singular_harmonic}
Let $(X^3,g)$ be an AC oriented Riemannian $3$-manifold with one end.
Then there are constants $c_1,c_2>0$ such that, for every $x\in X$,
there exists a distribution
\[
    \phi_x\in\bigl(C_c^\infty(X)\bigr)',
\]
represented by a smooth function on $X\setminus\{x\}$, satisfying
\[
    \Delta\phi_x=\delta_x
\]
in the sense of distributions and such that
	\begin{eqnarray}
	{\phi_x}|_{V(x)} = -\frac{c_1}{r} + O(1)\quad\text{and}\label{eq:asymp_behavior_phi_x_1}\\
	{\phi_x}|_{V(N)} = -\frac{c_2}{\vol(N)}\frac{1}{r} + O(r^{-2}),\label{eq:asymp_behavior_phi_x_2}
	\end{eqnarray} 
	where $r:=\dist_g(\,\cdot\,,x)$, and $V(x)$ and $V(N)$ denote a neighbourhood of $x$ and a neighbourhood of the end of $X$, respectively.
\end{proposition}

\underline{\textbf{Conventions}}.
Henceforth, unless otherwise stated, $(X^3,g)$ will denote an AC
oriented Riemannian $3$-manifold with one end, and $E$ will be a
$G$-bundle over $X$, where $G$ is a compact Lie group. The adjoint
bundle $\mathfrak g_E$ is equipped with the metric induced by the fixed
positive-definite Ad-invariant inner product on $\mathfrak g$ chosen at
the beginning of the Introduction. When $G=\SU(2)$, this is the
normalization
\[
    \langle a,b\rangle=-2\operatorname{tr}(ab),
    \qquad a,b\in\mathfrak{su}(2).
\] We denote by $\vol_g$ the Riemannian volume measure of $(X,g)$. We use
the standard normalization of the Hausdorff measures $\mathcal{H}^k$ on $(X,g)$; in particular, $\mathcal H^3=\vol_g$ as Borel measures on $(X^3,g)$, e.g., $\mathcal H^3=\vol_{g_E}$ on the Euclidean space $(\mathbb R^3,g_E)$, and $\mathcal H^0$ is the counting measure.

Except in Proposition~\ref{prop: scaling} and Remark~\ref{rem:higher-rank-obstruction}, whenever we restrict attention to monopoles we take $G=\SU(2)$. We denote by $c>0$ a generic
constant depending only on the geometry of $(X^3,g)$ and possibly on
the metric Lie algebra $(\mathfrak g,\langle\cdot,\cdot\rangle)$. Its
value may change from one occurrence to the next. Should $c$ depend on
further data, we indicate this by a subscript. We write $x\lesssim y$
for $x\leqslant cy$ and $O(x)$ for a quantity $y$ satisfying
$\lvert y\rvert\lesssim x$. The notation
$\{\cdot,\ldots,\cdot\}$ denotes a generic multilinear expression
bounded by $c$. Since $g$ has bounded geometry, we fix once and for all
a constant
$0<\rho_0(g)\ll_g\operatorname{inj}(X,g)$ satisfying
$\rho_0(\lambda^2g)=\lambda\rho_0(g)$ for all $\lambda>0$, and such that every ball $B_{\rho_0}(x)\subset(X^3,g)$ is strongly geodesically convex and geometrically uniformly controlled.

\section{Motivating examples}\label{sec: examples}

In this section, we collect a few examples that motivate the current work. The first of these consists of exploring the explicit Prasad-Sommerfield monopole in the limit when its mass is sent off to infinity. The second example uses Taubes' construction of multi-monopoles on $\mathbb{R}^3$ to produce sequences of charge $k\geq 1$ monopoles with unbounded masses, such that the corresponding zero set $Z$ is any a priori prescribed set of $l$ pairwise distinct points in $\mathbb R^3$, for any given $1\leqslant l\leqslant k$. Next, we include a simple general way to produce, from a given charge $k>1$ monopole, examples of sequences of charge $k$ monopoles in $\mathbb{R}^3$ with unbounded masses and for which the zero set $Z=\{0\}$ and the total cluster charge $K_0$ at the origin equals $k>1$. Finally, using the multi-monopole construction of \cite{oliveira2016monopoles} in the more general setting of an AC $3$-manifold $(X^3,g)$ with $b^2(X)=0$, we construct sequences of charge $k$ monopoles with unbounded masses whose zero set is any a priori prescribed set of $k$ pairwise distinct points in $X$. Section~\ref{sec:hierarchical-example} gives an additional charge four family in which several mass one profiles converge to the same point of $X$ while separating through a hierarchy of intermediate scales.

\subsection{The BPS Monopole}
In this section we shall write down the standard mass $m$ BPS monopole $(A_m, \Phi_m)$ on $\mathbb{R}^3$, constructed by Prasad and Sommerfield in \cite{BPS}. For any $m \in \mathbb{R}^+$, this has a unique zero $\Phi_m^{-1}(0)= \lbrace 0 \rbrace$ and is spherically symmetric. Obviously, by considering the sequence letting $m \rightarrow \infty$ we will have $Z = \lbrace 0 \rbrace$, however, the interesting thing of considering this specific example is that we shall be able to check the convergence to the delta function on $Z$ explicitly.

Write $\mathbb{R}^3 \backslash \lbrace 0 \rbrace \cong \mathbb{R}_+ \times \mathbb{S}^2$, and pullback from $\mathbb{S}^2 \cong \text{SU}(2)/\text{U}(1)$ the homogeneous bundle
\[
P = \text{SU}(2) \times_{\chi} \text{SU}(2),
\]
with $\chi: \text{U}(1) \rightarrow \text{SU}(2)$ the group homomorphism given by $\chi(e^{i \theta}) = \text{diag} (e^{i \theta}, e^{-i \theta})$. In this polar form, and actually working on the pullback to the total space of the radially extended Hopf bundle $\mathbb{R}^+ \times \text{SU}(2)$, the Euclidean metric can be written as
\[
g_E = dr^2 + 4 r^2 ( \omega_2 \otimes \omega_2+ \omega_3 \otimes \omega_3),
\]
where $r$ is the radial direction, i.e. the distance to the origin. Now fix the standard basis $\lbrace {S}_i \rbrace$ of $\mathfrak{su}(2)$ given by the Pauli matrices, and let $\omega_1, \omega_2 , \omega_3$ be the dual coframe. The $1$-form $ {S}_1 \otimes \omega^1 \in \Omega^1(\text{SU}(2), \mathfrak{su}(2))$ equips the Hopf bundle $\text{SU}(2) \rightarrow \mathbb{S}^2$ with an $\text{SU}(2)$-invariant connection, which in turn, induces a connection in $P$. Making use of Wang's theorem \cite{Wang1958}, one can write any other spherically symmetric connection on $\mathbb{R}^3 \backslash \lbrace 0 \rbrace$ as
\[
A= {S}_1 \otimes \omega_1 + a(r) ({S}_2 \otimes \omega^2 + {S}_3 \otimes \omega^3 ),
\]
for some function $a : \mathbb{R}^+ \rightarrow \mathbb{R}$. Similarly, viewing a Higgs field $\Phi(r)$ as a function in the total space with values in $\mathfrak{su}(2)$ one can show, see the Appendix in \cite{oliveira2014monopoles}, that any spherically symmetric Higgs field must be of the form $\Phi = \phi (r) \ {S}_1$, with $\phi  : \mathbb{R}^+ \rightarrow \mathbb{R}$ some function. A computation yields that 
\begin{eqnarray}\nonumber
F_A & = & 2( a^2 -1) {S}_1 \otimes \omega^{23} + \dot{a} ( {S}_2 \otimes dr \wedge \omega^2 + {S}_3 \otimes dr \wedge \omega^3 ) , \\ \nonumber
\nabla_A \Phi & = & \dot{\phi} \ {S}_1 \otimes dr + 2  a  \phi \  ({S}_2 \otimes \omega^3 - {S}_3 \otimes \omega^2 ),
\end{eqnarray}
with the dot denoting differentiation with respect to $r$. The energy density, as a function of $r$, is then
\begin{eqnarray}\label{eq:Energy_Density_Invariant}
e = \frac{(a^2-1)^2}{4r^4} + \frac{\dot{a}^2}{2r^2}+ \dot{\phi}^2 + \frac{2 a^2 \phi^2}{r^2}.
\end{eqnarray}
In this spherically symmetric setting, the monopole equations turn into the following system of ODEs:
\[
    \dot{\phi}
    =
    \frac{a^2-1}{2r^2},
    \qquad
    \dot a=2a\phi.
\]
Some particular solutions are given by the flat configurations
$(a,\phi)=(\pm1,0)$ and by the singular Dirac solutions
\[
    (a,\phi)=\left(0,c+\frac{1}{2r}\right),
    \qquad c\in\mathbb R.
\]
In particular, the singular solution with asymptotic Higgs field
$-mS_1$ is
\[
    (a,\phi)=\left(0,-m+\frac{1}{2r}\right).
\]
This sign convention agrees with the charge-one
Prasad--Sommerfield solution in
\cite{Jaffe1980}*{pp.~104--105}, whose Higgs field approaches the
negative unit radial section at infinity.

For the configuration $(A,\Phi)$ to extend smoothly over the origin,
one must instead impose
\[
    \phi(0)=0,
    \qquad
    a(0)=1.
\]
The resulting spherically symmetric Bogomolny system and its regular
solutions are analyzed in Appendix~A of
\cite{oliveira2014monopoles}; see in particular
equations~(A.2)--(A.3). In the Euclidean case, every nontrivial smooth
solution is of the form
\begin{align}
    \phi_m(r)
    &=
    \frac{1}{2}
    \left(
        \frac{1}{r}
        -
        \frac{2m}{\tanh(2mr)}
    \right),
    &
    a_m(r)
    &=
    \frac{2mr}{\sinh(2mr)},
\end{align}
where $m>0$ is the mass of the resulting monopole. Substitution into
\eqref{eq:Energy_Density_Invariant} gives
\begin{equation}\label{eq:BPS-energy-density}
e_m
=
\frac{1}{2}
\frac{
 \cosh^4(2mr)
 +(32m^4r^4-2)\cosh^2(2mr)
 -32m^3r^3\sinh(2mr)\cosh(2mr)
 +16m^4r^4+1
}{
 \sinh^4(2mr)\,r^4
}.
\end{equation}
Recall that in this case $Z=\lbrace0\rbrace$. The mass-$m$ BPS monopole
is obtained from the mass one monopole by dilation. Consequently, if
$D_{m^{-1}}(y):=m^{-1}y$, then
\[
    m^{-1}e_m\,\vol_{g_E}
    =
    (D_{m^{-1}})_*\left(e_1\,\vol_{g_E}\right).
\]
Indeed, $e_m(x)=m^4e_1(mx)$. Hence, for every
$f\in C_c^0(\mathbb R^3)$,
\[
    \int_{\mathbb R^3}f\,m^{-1}e_m\,\vol_{g_E}
    =
    \int_{\mathbb R^3}f(m^{-1}y)e_1(y)\,\vol_{g_E}.
\]
Since the mass one BPS monopole has energy $4\pi$, dominated convergence
gives
\[
    m^{-1}e_m\,\vol_{g_E}
    \rightharpoonup4\pi\delta_0
    \quad\text{as }m\to\infty.
\]

\begin{remark}
In this example $Z=\{0\}$. For every fixed $r>0$, the explicit formula
\eqref{eq:BPS-energy-density} gives
\[
    e_m(r)\longrightarrow\frac{1}{2r^4}
    \qquad\text{as }m\to\infty,
\]
and the convergence is smooth on compact subsets of
$\mathbb R^3\setminus\{0\}$. The limiting unnormalized density is the
energy density of the corresponding singular Dirac monopole. In
particular,
\[
    m^{-1}e_m\longrightarrow0
\]
smoothly on compact subsets of $\mathbb R^3\setminus\{0\}$, while the
mass-renormalized energy measures satisfy
\[
    m^{-1}e_m\,\vol_{g_E}
    \rightharpoonup4\pi\delta_0.
\]
\end{remark}

\subsection{\texorpdfstring{Sequences of Taubes' multi-monopoles on
$\mathbb{R}^3$ with prescribed $Z$}{Sequences of Taubes' multi-monopoles
on R3 with prescribed Z}}

We start by recalling the following Theorem of Taubes, see \cite{Jaffe1980}.

\begin{theorem}[Theorems 1.1 and 1.2 in \cite{Jaffe1980}]\label{thm:Taubes}
Let $k \geqslant2$. Then, there is $d_0>0$ and $c>0$ such that for any $y_1, \ldots , y_k \in \mathbb{R}^3$ with $d=\min_{j\neq l} \dist(y_j,y_l) >d_0$, there is a charge $k$, mass $1$, monopole $(A,\Phi)$ in $\mathbb{R}^3$. Furthermore, for $R=cd^{-1/2}$ we have that $\Phi^{-1}(0) \subset \cup_{i=1}^k B_R(y_i)$ and $\Phi \vert_{\partial B_R(y_i)}$ has degree $1$. In particular, $\Phi$ has a zero inside each of the balls $B_R(y_i)$, for $i=1, \ldots , k$.
\end{theorem}

We shall now use this construction to give several examples of the type considered in this paper.

\begin{proposition}
Let $1 \leqslant l \leqslant k$ be integers, $\lbrace x_1 , \ldots , x_l \rbrace\subseteq\mathbb{R}^3$ a subset of pairwise distinct points, and $\lbrace m_i \rbrace_{i\in\mathbb{N}}\subset\mathbb{R}^{+}$ an unbounded increasing sequence, i.e. $m_i \uparrow\infty$. Then, there is a sequence $\lbrace (A_i, \Phi_i) \rbrace_{i\in\mathbb{N}}$ of charge $k$, mass $m_i$ monopoles on $\mathbb{R}^3$ with zero set
\[
Z = \lbrace x_1 , \ldots , x_l \rbrace .
\]
\end{proposition}

\begin{remark}
In this construction, as will be evident during the proof, we have
$K_{x_j}=1$ for all $j=1,\ldots,l$. Moreover, when $l=k$, the
mass-renormalized energy densities satisfy
\[
    m_i^{-1}e(A_i,\Phi_i)\longrightarrow0
\]
smoothly on compact subsets of
$\mathbb R^3\setminus\{x_1,\ldots,x_k\}$; see
Lemmas~\ref{lem:small-energy-limit} and~\ref{lem:no-diffuse-limit}.
The mass-renormalized energy measures satisfy
\[
    m_i^{-1}e(A_i,\Phi_i)\vol_{g_E}
    \rightharpoonup
    4\pi\sum_{j=1}^k\delta_{x_j}.
\]
Thus every unit of charge is accounted for by one concentration point,
and no mass-renormalized energy remains away from
$\{x_1,\ldots,x_k\}$.

The case $l<k$ is precisely the case where $k-l$ monopoles escape
through the end, or run off to infinity. In the construction below,
for each $j=l+1,\ldots,k$ there is a zero
\[
    z_{i,j}
    \in
    B_{cm_i^{-1/2}}(m_i x_j).
\]
The centres of these balls leave every compact subset of
$\mathbb R^3$, and hence $z_{i,j}\to\infty$. These zeros therefore have
no convergent subsequence and do not contribute to $Z$.
\end{remark}

In the rest of this subsection we prove this result by using Theorem \ref{thm:Taubes} to construct the monopoles. Let $\lambda>0$ and consider the scaling map $s_{\lambda}(x)= \lambda^{-1} x$ for $x \in \mathbb{R}^3$. Recall that the Euclidean metric $g_E$ is invariant by scaling, i.e. $g_E = \lambda^{2} s_{\lambda}^* g_E$ for any such positive $\lambda$. Therefore, by Proposition \ref{prop: scaling}, if $(A,\Phi)$ is a charge $k$ mass $1$ monopole, we have that
\begin{equation}\label{eq:Examples_Rescaling}
(A_{\lambda} , \Phi_{\lambda}) = (s_{\lambda}^* A , {\lambda}^{-1} s_{\lambda}^* \Phi) 
\end{equation}
is a charge $k$, mass $\lambda^{-1}$ monopole.

When $k=1$, the proposition follows directly by translating and rescaling
the BPS monopole. We therefore assume below that $k\geqslant2$.
It is instructive to split the proof in two different cases, the case $l=k$ and the case $l<k$. We start with the first:

\begin{proof}[Case $l=k$]
If $k=1$, then $l=1$ and the result follows by translating the mass
$m_i$ BPS monopole so that its unique zero is $x_1$. Hence assume
$k\geqslant2$. We now construct a sequence of charge $k$, large mass
monopoles on $\mathbb{R}^3$ with prescribed
$Z= \lbrace x_1 , \ldots , x_k \rbrace$ being $k$ distinct points in
$\mathbb{R}^3$. After choosing such $k$ points, we fix a sequence of masses $m_i \rightarrow \infty$ and suppose, with no loss of generality, that $m_1 \gg 1$ so as to $m_1 \dist(x_j,x_l) > d_0$, for all distinct $j,l \in \lbrace 1, \ldots , k \rbrace$. Then, we can use Taubes' Theorem \ref{thm:Taubes} to construct  a sequence, labeled by $i$, of charge $k$, mass $1$ monopoles using in the construction the points $y^i_j=m_i x_j$ for $j=1 , \ldots , k$. Rescaling these as in equation \ref{eq:Examples_Rescaling} with $\lambda= m_i^{-1}$ we obtain a sequence of monopoles $(A_i, \Phi_i)$ with charge $k$, mass $m_i$ and
\begin{equation}\label{eq:Examples_Zeros_In_Balls}
\Phi_i^{-1}(0) \subset \bigcup_{j=1}^k B_{cm_i^{-1/2}}(x_j) , \ \ \mathrm{deg} ( \Phi_i \Big\vert_{\partial B_{cm_i^{-1/2}}(x_j)} ) =1 . 
\end{equation}
Now, recall from the definition of the zero set that
\[
Z=\bigcap_{n\geq 1}\overline{\bigcup_{i\geq n}\Phi_i^{-1}(0)}.
\]
Hence, as the sequence $\lbrace m_i \rbrace_i$ is increasing, it follows that
\[
\overline{\bigcup_{i\geq n}\Phi_i^{-1}(0)} \subseteq \bigcup_{j=1}^k \overline{B_{c m_n^{-1/2}}(x_j)},
\]
and thus
\[
Z\subseteq\bigcap_{n\geq 1}  \bigcup_{j=1}^k \overline{B_{c m_n^{-1/2}}(x_j)} =  \bigcup_{j=1}^k \bigcap_{n\geq 1}  \overline{B_{c m_n^{-1/2}}(x_j)} = \bigcup_{j=1}^k \lbrace x_j \rbrace.
\]
On the other hand, for every fixed $j\in\{1,\ldots,k\}$, the degree of the map $\Phi_i$ restricted to the sphere of radius $cm_i^{-1/2}$ equals $1$ and thus $\Phi_i$ has a zero 
\[
z_i \in B_{c m_i^{-1/2}}(x_{j}),
\]
for each $i\gg 1$. Since $m_i\uparrow\infty$, it follows that $z_i\to x_{j}$ as $i\to\infty$. Thus we get the reverse inclusion $\{x_1,\ldots,x_k\}\subseteq Z$, proving that indeed
\[
Z= \lbrace x_1 , \ldots , x_k \rbrace .
\]
\end{proof}

\begin{proof}[Case $l<k$]
We can modify the above construction in order to make $l<k$ of the monopoles ``escape to infinity''. We shall proceed as before and complete the prescribed points $x_1,\ldots,x_l$ to $k$ distinct points $\lbrace x_1,\ldots, x_k\rbrace \subseteq \mathbb{R}^3$, choosing $x_{l+1},\ldots,x_k$ nonzero. Then, we consider the charge $k$, mass $1$ monopole obtained through Taubes' Theorem \ref{thm:Taubes} using the points $y_j = m_i x_j$ for $j=1 , \ldots, l$ and the points $y_j = m_i^2 x_j$ for $j= l+1 , \ldots , k$.\footnote{By slight modification of this we can also let the points $y_j$ for $j>l$ go off to infinity at different rates.}
Then, rescaling this monopole as before, i.e. using equation \eqref{eq:Examples_Rescaling} with $\lambda=m_i^{-1}$, we obtain a mass $m_i$, charge $k$, monopole on $\mathbb{R}^3$. \begin{equation}\label{eq:Examples_Zeros_In_Balls_2}
\begin{split}
\Phi_i^{-1}(0)
&\subset
\left(
    \bigcup_{j=1}^l B_{cm_i^{-1/2}}(x_j)
\right)
\cup
\left(
    \bigcup_{j=l+1}^k
    B_{cm_i^{-1/2}}(m_i x_j)
\right),\\
\deg\left(
    \Phi_i\big|_{\partial B_{cm_i^{-1/2}}(x_j)}
\right)
&=1,
\qquad j=1,\ldots,l,\\
\deg\left(
    \Phi_i\big|_{\partial B_{cm_i^{-1/2}}(m_i x_j)}
\right)
&=1,
\qquad j=l+1,\ldots,k.
\end{split}
\end{equation}
Similarly to before we now have
\[
\overline{\bigcup_{i\geq n}\Phi_i^{-1}(0)} \subseteq \left( \bigcup_{j=1}^l \overline{B_{c m_n^{-1/2}}(x_j)} \right) \cup \left( \bigcup_{i \geq n} \bigcup_{j=l+1}^k \overline{B_{c m_n^{-1/2}}(m_i x_j)} \right),
\]
and the second union eventually leaves every compact subset of
$\mathbb R^3$. It then follows again from the same degree argument as
before that
\[
Z= \lbrace x_1 , \ldots , x_l \rbrace .
\]
\end{proof}

\subsection{\texorpdfstring{An example with $K_x>1$}{An example with
Kx greater than one}}

In the examples above, we have already seen that it is possible to have $\mathcal{H}^0(Z)<k$ by letting the monopoles ``escape through the end''. Another possibility is to have a concentration point $x\in Z$ with total cluster charge $K_x>1$. In this subsection, we give the simplest
such example.

Let $(A,\Phi)$ be a finite energy $\SU(2)$-monopole in $(\mathbb{R}^3,g_E)$ with mass $m\neq 0$ and charge $k > 1$. Since $g_E$ is scale-invariant, taking any null-sequence $\lambda_i\downarrow 0$ we get a corresponding sequence 
\[
(A_{\lambda_i},\Phi_{\lambda_i}):=(s_{\lambda_i}^{\ast}A , {\lambda_i}^{-1}s_{\lambda_i}^{\ast}\Phi)
\] of monopoles in $(\mathbb{R}^3,g_E)$ with masses $m_i:=m\lambda_i^{-1}\to\infty$. For this sequence, $Z=\{0\}$ and the unique concentration point has
total cluster charge $K_0=k>1$. In fact, there is a single Euclidean
profile, namely the original charge $k$ monopole rescaled to mass one.

\subsection{\texorpdfstring{Sequences of monopoles with prescribed $Z$ on
any AC $3$-manifold with $b^2(X)=0$}{Sequences of monopoles with prescribed
Z on AC 3-manifolds}}

Let $(X,g)$ be an AC oriented Riemannian $3$-manifold with $b^2(X)=0$, and let $k \in \mathbb{Z}_{>0}$. For a real number $m>0$, denote by $\mathcal{M}_{k,m}$ the moduli space of mass $m$ and charge $k$ monopoles on $(X,g)$. In this setting, the main result of \cite{oliveira2016monopoles} yields

\begin{theorem}[{\cite{oliveira2016monopoles}*{Theorem~1}}]\label{thm: Monopoles_Examples}
There exists $\mu\in\mathbb R$ such that, for every $m\geqslant\mu$, if
\[
X^k(m)
:=
\left\{
(x_1,\ldots,x_k)\in X^k:
\dist(x_i,x_j)>4m^{-1/2}\ \text{for }i\neq j
\right\}
\]
and
\[
\check{\mathbb T}^{k-1}
:=
\left\{
(e^{i\theta_1},\ldots,e^{i\theta_k})\in\mathbb T^k:
 e^{i(\theta_1+\cdots+\theta_k)}=1
\right\},
\]
then there is a local diffeomorphism onto its image
\begin{equation}\label{eq:Map}
h_m:
X^k(m)\times H^1(X,\mathbb S^1)\times\check{\mathbb T}^{k-1}
\longrightarrow\mathcal M_{k,m}.
\end{equation}
\end{theorem}

In order to use this theorem we shall fix once and for all $\alpha \in H^1(X, S^1)$ and $\theta = (e^{i\theta_1} , \ldots , e^{i \theta_k} ) \in \mathbb{T}^k$ satisfying $e^{i (\theta_1 + \ldots + \theta_k)}=1$. Then, we choose any $k$ pairwise distinct points $(x_1,...,x_k) \in X^k$ and take an increasing sequence of positive real numbers $m_i \uparrow \infty$. If $k\geqslant2$, we require $m_1 > \max\{16(\min_{j\neq l} \dist(x_j,x_l))^{-2},\mu\}$; if $k=1$, we require only $m_1>\mu$. We consider the monopoles
\[
(A_i , \Phi_i) = h_{m_i}((x_1, \ldots ,x_k), \alpha , \theta).
\]
Then, using the results of \cite{oliveira2016monopoles} we have the following

\begin{proposition}
The zero set
\[
Z=\bigcap_{n\geq 1}\overline{\bigcup_{i\geq n}\Phi_i^{-1}(0)},
\]
of the family of monopoles $(A_i, \Phi_i)$ defined above is precisely the set of points $ \lbrace x_1 , \ldots , x_k \rbrace$.
\end{proposition}
\begin{proof}
It suffices to prove that the zeros of the monopole $(A_i, \Phi_i)$ are contained in balls of radius $O(m_i^{-1/2})$ around the points $\lbrace x_1 , \ldots , x_k \rbrace$, i.e. for sufficiently large $i$ we have
\begin{equation}\label{eq:Zeroes_Phi}
\Phi_i^{-1}(0) \subset \bigcup_{j=1}^k B_{10 m_i^{-1/2}}(x_j)\quad\text{and}\quad \Phi_i^{-1}(0)\cap B_{10 m_i^{-1/2}}(x_j)\neq\emptyset,\quad\forall j\in\{1,\ldots,k\}.
\end{equation}
Indeed, if we prove this assertion, then, as the sequence $\lbrace m_i \rbrace$ is increasing, on the one hand note that
\[
\overline{\bigcup_{i\geq n}\Phi_i^{-1}(0)} \subseteq \bigcup_{j=1}^k \overline{B_{10 m_n^{-1/2}}(x_j)},
\]
and thus
\[
Z\subseteq\bigcap_{n\geq 1}  \bigcup_{j=1}^k \overline{B_{10 m_n^{-1/2}}(x_j)} =  \bigcup_{j=1}^k \bigcap_{n\geq 1}  \overline{B_{10 m_n^{-1/2}}(x_j)} = \bigcup_{j=1}^k \lbrace x_j \rbrace.
\] On the other hand, for every fixed $j_0\in\{1,\ldots,k\}$, we can find $z_i\in\Phi_i^{-1}(0)\cap B_{10 m_i^{-1/2}}(x_{j_0})$ for each $i\gg 1$. Since $m_i\uparrow\infty$, it follows that $z_i\to x_{j_0}$ as $i\to\infty$. Thus we get the reverse inclusion $\{x_1,\ldots,x_k\}\subseteq Z$, and equality follows as claimed.

We are thus left with proving the assertion \eqref{eq:Zeroes_Phi}, which is done in the Appendix \ref{appendix: A}.
\end{proof}

\subsection{A cluster with several separation scales}
\label{sec:hierarchical-example}

\noindent\textbf{New in version~5.}
We now give an example, new in version~5, which exhibits exactly the phenomenon
missed by the fixed centre argument.  All the monopoles collapse to a single
point in the original Euclidean space, and hence contribute to one point
of $S$, but after rescaling by the mass they separate to infinity.  The
centres may moreover possess a nontrivial hierarchy of intermediate
separation scales.  This hierarchy can be represented by a rooted tree of
clusters; nevertheless, only its leaves are finite mass Euclidean
monopoles.  The internal vertices do not represent secondary bubbles.

We use the well-separated multi-monopoles constructed by Taubes; see
\cite{Jaffe1980} and the formulation recalled in
Theorem~\ref{thm:Taubes}.  For fixed charge $k$, if
$y_1,\ldots,y_k\in\mathbb R^3$ have minimum pairwise separation
$d>d_0$, the construction gives a monopole of mass one and charge $k$ whose
zeros lie in mutually disjoint balls about the $y_j$.  The normalized
Higgs field has degree one around each of these balls.

\begin{proposition}[A two-level cluster over one concentration point]
\label{prop:hierarchical-cluster-example}
Let $m_i\to\infty$, and choose two positive sequences $L_i,M_i$ such that
\begin{equation}\label{eq:hierarchical-scales}
    1\ll L_i\ll M_i\ll m_i.
\end{equation}
There exists a sequence $(A_i,\Phi_i)$ of monopoles of mass $m_i$ and
charge four on $\mathbb R^3$ for which
\[
    S=Z=\{0\},
    \qquad
    \mu_i\wto16\pi\delta_0,
\]
and four sequences of centres $p_{i,1},\ldots,p_{i,4}\to0$ with the
following properties:
\begin{enumerate}
\item[(i)] the two pairs $\{p_{i,1},p_{i,2}\}$ and
$\{p_{i,3},p_{i,4}\}$ have internal separation comparable to
$L_i m_i^{-1}$;
\item[(ii)] the separation between these two pairs is comparable to
$M_i m_i^{-1}$;
\item[(iii)] rescaling at the mass scale about each $p_{i,j}$ gives a mass one,
charge one Euclidean monopole;
\item[(iv)] after passage to a subsequence and change of gauge, a fixed centre
rescaling about the origin detects only the first of these bubbles,
whereas the full concentration weight at the origin is $16\pi$.
\end{enumerate}
Thus the concentration at the single point $0$ contains four mutually
separating mass one bubbles and has a two-level tree of separation scales.
There is no secondary finite mass bubble at either intermediate level.
\end{proposition}

\begin{proof}
Fix orthonormal vectors $e_1,e_2\in\mathbb R^3$ and set
\[
\begin{aligned}
    y_{i,1}&=0,
    &\qquad y_{i,2}&=L_i e_1,\\
    y_{i,3}&=M_i e_2,
    &\qquad y_{i,4}&=y_{i,2} + y_{i,3}=L_i e_1 + M_i e_2.
\end{aligned}
\]
Their minimum pairwise separation is $L_i$, which tends to infinity.
Taubes' construction therefore gives, for all sufficiently large $i$, a
monopole $(B_i,\Psi_i)$ of mass one and charge four associated with these four
centres.  If $c$ is the constant in the well-separated construction, then
all zeros of $\Psi_i$ lie in the union of the balls
\[
    B_{\rho_i}(y_{i,j}),
    \qquad
    \rho_i=cL_i^{-1/2},
    \qquad 1\leqslant j\leqslant4,
\]
and the normalized Higgs field has degree one on the boundary of every
such ball.  In particular, choose a zero
$z_{i,j}\in B_{\rho_i}(y_{i,j})$ for each $j$.

Let $D_i:\mathbb R^3\to\mathbb R^3$ be $D_i(x)=m_i x$ and define
\begin{equation}\label{eq:hierarchical-example-rescaling}
    A_i=D_i^*B_i,
    \qquad
    \Phi_i=m_iD_i^*\Psi_i.
\end{equation}
Then $(A_i,\Phi_i)$ has mass $m_i$ and charge four.  Put
\[
    p_{i,j}:=m_i^{-1}z_{i,j}.
\]
Since $\rho_i=o(L_i)$, the relative positions of these points satisfy
\begin{align*}
    m_i\abs{p_{i,1}-p_{i,2}}
       &=L_i+o(L_i),
    &
    m_i\abs{p_{i,3}-p_{i,4}}
       &=L_i+o(L_i),\\
    m_i\dist\bigl(\{p_{i,1},p_{i,2}\},
                         \{p_{i,3},p_{i,4}\}\bigr)
       &=M_i+o(M_i).
\end{align*}
The last relation uses $L_i=o(M_i)$.  Moreover,
$\max_j\abs{p_{i,j}}=O(M_i/m_i)\to0$, so every zero of $\Phi_i$
converges to the origin.  Conversely, the degree one condition provides a
zero in each of the four shrinking balls.  Hence $Z=\{0\}$ and, by
Theorem~\ref{thm: Main_Monopoles}, also $S=\{0\}$.

The mass scale rescaling of $(A_i,\Phi_i)$ about $p_{i,j}$ is the
translation of $(B_i,\Psi_i)$ by the zero $z_{i,j}$.  The corrected
$\varepsilon$-regularity estimate and local compactness give a nontrivial
mass one Euclidean monopole after passing to a subsequence.  Distinct
centres separate to infinity in these rescaled coordinates.  The degree
of the normalized Higgs field on a fixed large sphere about $z_{i,j}$ is
one for all sufficiently large $i$: the degree is one on
$\partial B_{\rho_i}(y_{i,j})$, and there are no further zeros before a
sphere separating this centre from the other three.  Therefore every
pointed limit has charge one.  The four bubbles exhaust the total charge,
so $K_0=4$, and Theorem~\ref{thm: Main_Monopoles} gives
\[
    \mu_i\wto4\pi K_0\delta_0=16\pi\delta_0.
\]

Finally, because $y_{i,1}=0$ and $z_{i,1}\to0$, rescaling the original
sequence about the fixed point $0$ at the mass scale produces the same
pointed limit as rescaling about $p_{i,1}$.  The other three centres occur
at distances of order $L_i$ or $M_i$ in the coordinates rescaled by the mass and
therefore leave every compact subset of $\mathbb R^3$.  The fixed centre
limit consequently has charge one, even though the concentration weight
at the corresponding point of $S$ is $16\pi$.

At the intermediate original scales $L_i/m_i$ and $M_i/m_i$, the rescaled
Higgs masses are respectively $L_i$ and $M_i$, both of which tend to
infinity.  These scales record the splitting of clusters but cannot
produce additional finite mass Euclidean monopoles.  This proves all the
claims.
\end{proof}

\begin{remark}[Arbitrary finite separation trees]
\label{rem:arbitrary-separation-trees}
The preceding construction realizes an arbitrary finite rooted hierarchy
of clusters.  Given a rooted tree with $k$ leaves and height $h$, choose
scales
\[
    1\ll L_i^{(1)}\ll L_i^{(2)}\ll\cdots
    \ll L_i^{(h)}\ll m_i
\]
and choose the $k$ Taubes centres so that two leaves which first separate
at level $s$ have mutual distance comparable to $L_i^{(s)}$.  This can be
done by assigning distinct fixed displacement vectors to the branches at
each level and summing the corresponding vectors weighted by
$L_i^{(s)}$.  The minimum pairwise separation tends to infinity, while all
centres are $o(m_i)$ from the origin.  After rescaling to mass $m_i$, all
zeros converge to $0$, and the $k$ leaves yield $k$ mass one bubbles of charge one over that point.  The internal vertices describe only the
hierarchy of separations.  Rescaling at an internal scale
$L_i^{(s)}/m_i$ produces Higgs mass $L_i^{(s)}\to\infty$, so it does not
produce an additional finite mass bubble.  Thus one may draw a tree of
clusters, but the complete bubbling object is still a finite collection
of same scale Euclidean monopoles.
\end{remark}

\section{AC and mass dependent version of Taubes' small Higgs field estimates}\label{sec: taubes}

Let $(X^3,g)$ be an AC oriented Riemannian $3$-manifold with one end, let $E$ be an $\SU(2)$-bundle over $X$ and let $(A,\Phi)$ be a finite energy $\SU(2)$ monopole on $(X,g)$ with charge $k\neq 0$ and mass $m\neq 0$. In \cite{taubes2014magnetic} Taubes poses and addresses the following question, in the case where $(X^3,g)$ is the Euclidean space $(\mathbb{R}^3,g_E)$:
\begin{question}
	What is the \emph{largest} radius of a ball in $X$ that contains \emph{only} points where $\lvert\Phi\rvert\ll m$?
\end{question}
Below, we shall prove the analogue of Taubes' result \cite{taubes2014magnetic}*{Theorem~1.2} in the case of more general Yang--Mills--Higgs configurations on an asymptotically conical $(X^3,g)$ with one end. 
\begin{theorem}\label{thm: taubes_estimate}
Let $(X^3,g)$ be an oriented AC Riemannian $3$-manifold with one end,
let $E$ be a $G$-bundle over $X$, and let $(A,\Phi)$ be a finite mass
YMH configuration on $E$. Given $\delta\in(0,1)$ and
$\Lambda\in(0,\infty)$, there is a constant $m_{\ast}>0$, depending
only on $g$, $\Lambda$, and $\delta$, with the following significance.
Let $m>m_{\ast}$ denote the mass of $(A,\Phi)$ and suppose that
\[
    m^{-1}\|\textnormal{d}_A\Phi\|_{L^2(X)}^2\leqslant\Lambda.
\]
Then
\[
    r_{\delta}(x)
    :=
    \sup\left\{
        r\in[0,\infty):
        \sup_{B_r(x)}\abs{\Phi}<m\delta
    \right\}
\]
satisfies the upper bound\footnote{It is clear from our proof that this
upper bound is not sharp, but for our purposes it suffices to know that
$\displaystyle
r_\delta(x)\lesssim\Lambda\bigl(m(1-\delta)\bigr)^{-1}$.}
\begin{equation}\label{eq: radius_estimate_YMH}
    r_{\delta}(x)
    \leqslant
    \frac{4\Lambda c_1}{m(1-\delta)c_2},
\end{equation}
where $c_1,c_2>0$ are the constants of
Proposition~\ref{prop: singular_harmonic}. In particular, if
$(A,\Phi)$ is an $\SU(2)$ monopole of charge $k$, then
\begin{equation}\label{eq: radius_estimate_monopoles}
    r_{\delta}(x)
    \leqslant
    \frac{16\pi k c_1}{m(1-\delta)c_2}.
\end{equation}
\end{theorem}
\begin{proof}
	Suppose for contradiction that, for every $m_{\ast}>0$ depending only
on the indicated data, there exists a finite mass YMH configuration
$(A,\Phi)$ of mass $m>m_{\ast}$ such that
\[
    s:=\frac{4\Lambda c_1}{m(1-\delta)c_2}<r_\delta(x).
\]
	Let $\phi_x$ be the harmonic function on $X\setminus\{x\}$ given by Proposition~\ref{prop: singular_harmonic}, and let $\phi_0:= c_2^{-1}2\Lambda\phi_x$. Then, for small enough $r = \dist(x , \cdot)$, equation \eqref{eq:asymp_behavior_phi_x_1} yields
	\begin{align}\label{eq:expansion_phi_0}
	\phi_0|_{V(x)} \geq -\frac{m(1-\delta)s}{2r} + c_{\Lambda},
	\end{align} 
	for some constant $c_{\Lambda}\in\mathbb{R}$, depending only on $g$ and $\Lambda$. 
	
	Now, as $s$ is inversely proportional to $m$, there is $m_{\ast}>0$, depending only on $g$, $\Lambda$ and $\delta$, so that the expansion \eqref{eq:expansion_phi_0} is valid for $r=s$. At this point, it is convenient to further define the harmonic function on $X \backslash \lbrace x \rbrace$ given by $\phi:= \phi_0 + m$. Then, by possibly increasing $m_{\ast}$ so that $m_{\ast} > - 2c_{\Lambda}(1-\delta)^{-1}$, we have
	\[
    \inf_{\partial B_s(x)}\phi
    \geq
    -\frac{m(1-\delta)}{2}+c_\Lambda+m
    >
    m\delta
    >
    \sup_{\partial B_s(x)}\abs{\Phi}.
\]
	where the last inequality follows from $s<r_\delta(x)$. Then, the previous inequality, and the fact that both the harmonic function $\phi$ and the subharmonic function $\vert \Phi \vert$ converge to $m$ along the end show that
	\[
	\vert \Phi \vert < \phi\quad\text{in $X\setminus B_{s}(x)$}.
	\] On the other hand, by \eqref{eq: asymp_behavior_YMH} and \eqref{eq:asymp_behavior_phi_x_2}, we have
	\begin{align*}
	\lvert\Phi\rvert &\geq m- \Lambda\vol(N)^{-1}\rho^{-1} + o(\rho^{-1}),\quad\text{and}\\
	\phi &= m - 2\Lambda\vol(N)^{-1}\rho^{-1} + o(\rho^{-1}),
	\end{align*} as $\rho\to\infty$. Putting these together, we conclude that $\Lambda\vol(N)^{-1}\geq 2\Lambda\vol(N)^{-1}$, hence a contradiction. This completes the proof of \eqref{eq: radius_estimate_YMH}. The case of monopoles \eqref{eq: radius_estimate_monopoles} then follows by using the energy formula \eqref{eq: Energy_Formula}. 
\end{proof}

\medskip
\noindent\textbf{New in version~5.}
There is an alternative estimate which requires only finite energy and
is valid for every compact structure group. Let $(A,\Phi)$ be a finite
energy configuration on $E$, let $m>0$ be its mass, and suppose that
\[
    m^{-1}\norm{\textnormal{d}_A\Phi}_{L^2(X)}^2
    \leqslant\Lambda.
\]
If
\[
    \sup_{B_r(x)}\abs{\Phi}\leqslant m\delta,
\]
then $m-\abs{\Phi}\geqslant m(1-\delta)$ on $B_r(x)$. Hence
\[
    m(1-\delta)\vol(B_r(x))^{1/6}
    \leqslant
    \norm{m-\abs{\Phi}}_{L^6(X)}.
\]
By the uniform volume lower bound
\[
    \vol(B_r(x))\gtrsim r^3
\]
for AC $3$-manifolds
\cite{van2009regularity}*{Corollary~2.8}, together with
\eqref{ineq: crit_Sobolev}, we obtain
\[
    r^{1/2}
    \lesssim
    \frac{\norm{\textnormal{d}_A\Phi}_{L^2(X)}}
         {m(1-\delta)}.
\]
Taking the supremum over all such radii gives
\begin{equation}\label{eq:finite-energy-radius-estimate}
    r_\delta(x)
    \leqslant
    \frac{C_g\Lambda}{m(1-\delta)^2},
\end{equation}
for a constant $C_g>0$ depending only on the geometry of $(X,g)$.
Although \eqref{eq:finite-energy-radius-estimate} has a weaker
dependence on $\delta$ than \eqref{eq: radius_estimate_YMH}, it is
sufficient for every application below, where $\delta$ is fixed. In
particular, all those applications hold for finite energy configurations
with arbitrary compact structure group, without the additional
hypothesis of Definition~\ref{def: finite_mass}.

\begin{remark}
Under the hypotheses of Theorem~\ref{thm: taubes_estimate}, whenever
$m>m_{\ast}$ and
$m^{-1}\|\textnormal{d}_A\Phi\|_{L^2(X)}^2\leqslant\Lambda$, one has
\[
    \sup_{\partial B_r(x)}\lvert\Phi\rvert
    \geqslant
    \sup_{B_r(x)}\lvert\Phi\rvert
    \geqslant m\delta
\]
for every
\[
    r>c\Lambda m^{-1}(1-\delta)^{-1},
\]
where $c>0$ depends only on $g$. Here the first inequality follows from
the maximum principle.

Alternatively, if $(A,\Phi)$ is merely a finite energy configuration
with arbitrary compact structure group, then the same conclusion holds
for every
\[
    r>C_g\Lambda m^{-1}(1-\delta)^{-2}
\]
by \eqref{eq:finite-energy-radius-estimate}.
\end{remark}

\section{\texorpdfstring{$\varepsilon$}{epsilon}-regularity estimate}
\label{sec: e-regularity}

\noindent\textbf{Corrected in version~5.}
The Bochner estimate used in the published version is valid but too
coarse for the argument uniform in the mass, because it bounds the
favourable Higgs commutator terms by a positive multiple of
$\lvert\Phi\rvert^2e(A,\Phi)$. Retaining these terms with their natural
coercive sign yields the refined estimate and the corrected
$\varepsilon$-regularity proof below. The rescaling and local-estimate
strategy is parallel to the a priori analysis for the scaled
Yang--Mills--Higgs functional with Higgs self-interaction in
\cite{cheng2025su2}*{Sections~3.2--3.5}; here the potential term is absent,
and we record a self-contained argument adapted to the ordinary
Yang--Mills--Higgs equations.
\medskip

Set
\[
    e(A,\Phi)
    :=
    \frac12\left(\abs{F_A}^2+\abs{\text{d}_A\Phi}^2\right).
\]

\begin{lemma}[Refined coercive Bochner estimate]
\label{lem: Bochner_estimate}
Let $(A,\Phi)$ be a YMH configuration on a domain in an
oriented Riemannian manifold. Then
\begin{equation}\label{eq:bochner-energy-v5}
\begin{split}
    \Delta e(A,\Phi)
    \leqslant{}&
    C\left(\abs{R^g}e(A,\Phi)+e(A,\Phi)^{3/2}\right)
    \\
    &-
    \abs{[F_A,\Phi]}^2
    -
    \abs{[\text{d}_A\Phi,\Phi]}^2,
\end{split}
\end{equation}
where $C$ depends only on the dimension and the structure constants of the
Lie algebra. In particular, when $\abs{R^g}$ is bounded,
\[
    \Delta e(A,\Phi)
    \lesssim
    e(A,\Phi)+e(A,\Phi)^{3/2},
\]
with no dependence on $\abs\Phi$.
\end{lemma}

\begin{proof}
The relevant Bochner formulae and the YMH equations give
the following schematic identities:
\begin{align*}
\frac12\Delta\abs{\text{d}_A\Phi}^2
={}&
-\abs{\nabla_A(\text{d}_A\Phi)}^2
-\abs{[\text{d}_A\Phi,\Phi]}^2
+R^g*\text{d}_A\Phi*\text{d}_A\Phi
+F_A*\text{d}_A\Phi*\text{d}_A\Phi,
\\
\frac12\Delta\abs{F_A}^2
={}&
-\abs{\nabla_A F_A}^2
-\abs{[F_A,\Phi]}^2
+R^g*F_A*F_A
+F_A*F_A*F_A
+F_A*\text{d}_A\Phi*\text{d}_A\Phi.
\end{align*}
These identities are recorded in
\cite{fadel2023asymptotics}*{Lemma~4.1}; the underlying standard
Bochner--Weitzenb\"ock formulae may be found in
\cite{bourguignon1981stability}*{Theorems~3.2 and~3.10}. The signs of the
Higgs terms follow from Ad-invariance:
\[
    \ip{[[\eta,\Phi],\Phi]}{\eta}
    =-\abs{[\eta,\Phi]}^2,
    \qquad
    \eta=F_A,\ \text{d}_A\Phi.
\]
Adding the identities and estimating the remaining algebraic terms proves
\eqref{eq:bochner-energy-v5}.
\end{proof}

\begin{theorem}[Corrected $\varepsilon$-regularity estimate]
\label{thm: e-regularity}
There are constants $m_0,\rho_0,\varepsilon_0,C_0>0$, depending only on
the bounded geometry of $(X,g)$ and the structure group, with the following
property. Let $(A,\Phi)$ be a YMH configuration and let
$m\geqslant m_0$ be a positive parameter. For every $R>0$, if
\[
    0<r\leqslant\min\{Rm^{-1},\rho_0\},
    \qquad
    \varepsilon
    :=m^{-1}\mathscr E_{B_r(x)}(A,\Phi)<\varepsilon_0,
\]
then
\begin{equation}\label{eq: e-regularity}
    \sup_{B_{r/4}(x)}m^{-1}e(A,\Phi)
    \leqslant
    C_Rr^{-3}\varepsilon,
    \qquad
    C_R:=C_0\max\{1,R^3\}.
\end{equation}
\end{theorem}

\begin{proof}
With respect to $g_m:=m^2g$, the pair $(A,m^{-1}\Phi)$ is a
YMH configuration and
\[
    \mathscr E^{m}_{B^m_{rm}(x)}(A,m^{-1}\Phi)
    =m^{-1}\mathscr E_{B_r(x)}(A,\Phi)<\varepsilon_0,
\]
where $B^m$ and $\mathscr E^m$ denote balls and energy computed with $g_m$.
The metrics $g_m$ have bounded-geometry constants uniform in
$m\geqslant m_0$ on balls of bounded radius. It is therefore enough to
prove the estimate after this rescaling, with $m=1$ and $r\leqslant R$.
The remaining argument is the standard Heinz trick; compare
\cite{walpuski2017compactness}*{Appendix~A}.

On $\overline{B}_{r/2}(x)$ consider
\[
    \theta(y)
    :=
    \left(\frac r2-d(x,y)\right)^3e(A,\Phi)(y).
\]
Let $y_0$ be a maximum point and set
\[
    M:=\theta(y_0),
    \qquad
    e_0:=e(A,\Phi)(y_0),
    \qquad
    s_0:=\frac12\left(\frac r2-d(x,y_0)\right).
\]
Then $e(A,\Phi)\leqslant8e_0$ on $B_{s_0}(y_0)$. By
Lemma~\ref{lem: Bochner_estimate} and the local mean-value
inequality (see, for example,
\cite{gilbarg2001elliptic}*{Theorem~9.20}), for every
$0<s\leqslant\min\{s_0,r_1\}$,
\begin{equation}\label{eq:heinz-local-estimate}
    e_0
    \leqslant
    C\left(
       s^{-3}\varepsilon+s^2(e_0+e_0^{3/2})
    \right),
\end{equation}
where $r_1>0$ and $C$ are uniform for the rescaled metrics.

If $e_0\leqslant1$, choose
$s=\min\{s_0,s_1\}$, with $s_1\leqslant r_1$ fixed sufficiently small to
absorb the final term in \eqref{eq:heinz-local-estimate}. This gives either
$M\lesssim\varepsilon$ or
$M\lesssim R^3\varepsilon$. If $e_0>1$, choose
\[
    s=\min\{s_0,c_1e_0^{-1/4}\},
\]
with $c_1>0$ fixed sufficiently small. If $s=s_0$, then
$M\lesssim\varepsilon$; otherwise
\[
    e_0\lesssim e_0^{3/4}\varepsilon,
\]
which is impossible once $\varepsilon_0$ is sufficiently small. Hence
$M\lesssim\max\{1,R^3\}\varepsilon$. Since
$\theta\geqslant(r/4)^3e$ on $B_{r/4}(x)$, this proves the estimate for
$m=1$; scaling back gives
\eqref{eq: e-regularity}.
\end{proof}

The normalization in the preceding theorem also changes the statement of
\cite{fadel2019limit}*{Theorem~6.1}: its mass hypothesis must be
$m\geqslant m_0$, and, with the notation used there, the correct
small energy threshold is
\begin{equation}\label{eq:corrected-lower-Higgs-threshold}
    \varepsilon_{\delta,R}
    :=
    \min\left\{(2C_R)^{-1}R\delta^2,\varepsilon_0\right\}.
\end{equation}
Indeed, since
$e=\frac12(\abs{F_A}^2+\abs{\text{d}_A\Phi}^2)$, one has
$\abs{\text{d}_A\Phi}\leqslant\sqrt2e^{1/2}$. The proof of that theorem
then gives the stated lower bound $\abs\Phi>m\delta/2$. Consequently, the mass hypothesis in
\cite{fadel2019limit}*{Corollary~6.1} must be
$m>\max\{m_{\ast},m_0\}$, and its constant may be chosen as
\[
    \varepsilon_\Lambda
    =
    \min\left\{
       8^{-1}C_{R_\Lambda}^{-1}R_\Lambda,
       \varepsilon_0
    \right\}.
\]

\section{An interior lower bound on the Higgs field}\label{sec: lowerbound}

\noindent\textbf{Update in version~5.}
The constants in the local Higgs lower-bound argument have been adjusted
to agree with the corrected normalization of the energy and the revised
$\varepsilon$-regularity theorem.
\medskip

The next result is a consequence of the previous $\varepsilon$-regularity estimate and will prove to be useful in analysing large mass Yang--Mills--Higgs configurations.
\begin{theorem}\label{thm: zeros_in_sigma}
	Let $(X^3,g)$ be an AC oriented Riemannian $3$-manifold with one end, let $E$ be a $G$-bundle over $X$, and let $(A,\Phi)$ be a YMH configuration on $E$ with mass
$m\geqslant m_0$. Given $\delta\in (0,1)$ and $R>0$, set 
	\[
	\varepsilon_{\delta,R}:=\min\left\{(2C_{R})^{-1}R\delta^2,\varepsilon_0\right\}.
	\]
	Let $x\in X$. If $r:=Rm^{-1}\leqslant\rho_0$ and $\displaystyle\sup_{\partial\overline{B}_{\frac{r}{4}}(x)}\lvert\Phi\rvert\geq m\delta$, then
	\[
	m^{-1}\mathscr{E}_{B_{r}(x)}(A,\Phi)<\varepsilon_{\delta,R}\quad\Longrightarrow\quad\lvert\Phi\rvert > \frac{m\delta}{2}\quad\text{on}\quad B_{\frac{{r}}{4}}(x).
	\] Here $C_{R}$ and $\varepsilon_0$ are the constants given by Theorem \ref{thm: e-regularity}.
\end{theorem}
\begin{proof}
	Fix $q\in\partial\overline{B}_{\frac{r}{4}}(x)$ such that the
restriction of $\lvert\Phi\rvert$ to
$\partial\overline{B}_{\frac{r}{4}}(x)$ attains its maximum at $q$.
For any $p\in B_{\frac{r}{4}}(x)$, let $\gamma_p$ be the concatenation
of a minimizing geodesic from $p$ to $x$ and a minimizing radial
geodesic from $x$ to $q$. Then
\[
    \gamma_p\subset\overline{B}_{\frac{r}{4}}(x),
    \qquad
    L(\gamma_p)\leqslant\frac r2.
\] Thus, using the fundamental theorem of calculus and Kato's inequality we get:
	\begin{equation}\label{eq: est_1}
	\lvert\Phi\rvert(q) - \lvert\Phi\rvert(p)\leqslant\left\vert\int_{\gamma_p}\text{d}\lvert\Phi\rvert\right\vert\leqslant\int_{\gamma_p}\lvert\text{d}_A\Phi\rvert\leqslant {\frac{r}{2}}\sup_{B_{\frac{r}{4}}(x)}\lvert\text{d}_A\Phi\rvert.
	\end{equation} On the other hand, by the $\varepsilon_0$-regularity estimate (Theorem \ref{thm: e-regularity}), the hypothesis $ m^{-1}\mathscr{E}_{B_r(x)}(A,\Phi)<\varepsilon_{\delta,R}\leqslant\varepsilon_0$ implies that
	\begin{equation}\label{eq: est_2}
	\sup_{B_{\frac{r}{4}}(x)}\lvert\text{d}_A\Phi\rvert\leqslant\sqrt{2}\sup_{B_{\frac{r}{4}}(x)}e(A,\Phi)^{\frac{1}{2}}< (2C_{R})^{\frac{1}{2}}r^{-\frac{3}{2}}m^{\frac{1}{2}}\varepsilon_{\delta,R}^{\frac{1}{2}}.
	\end{equation} Putting \eqref{eq: est_1} and \eqref{eq: est_2} together and using the definitions of $r$ and $\varepsilon_{\delta,R}$ along with the lower bound on $\lvert\Phi\rvert(q)=\displaystyle\sup_{\partial\overline{B}_{\frac{r}{4}}(x)}\lvert\Phi\rvert$ gives the statement.
\end{proof}
Combining the above result with the finite energy radius estimate
\eqref{eq:finite-energy-radius-estimate}, we obtain:
\begin{corollary}\label{cor: zeros_in_sigma}
Let $(X^3,g)$ be an oriented AC Riemannian $3$-manifold with one end,
let $E$ be a $G$-bundle over $X$, where $G$ is compact, and let
$\Lambda\in(0,\infty)$. Then there are constants
$R_{\Lambda}>0$ and $\varepsilon_{\Lambda}>0$ such that the following
holds. Let $(A,\Phi)$ be a finite energy YMH configuration of mass
$m>m_0$ satisfying
\[
    m^{-1}\|\textnormal{d}_A\Phi\|_{L^2(X)}^2
    \leqslant\Lambda.
\]
If $r:=R_{\Lambda}m^{-1}\leqslant\rho_0$, then
\begin{equation}\label{eq: zeros_in_sigma}
    m^{-1}\mathscr{E}_{B_r(x)}(A,\Phi)
    <\varepsilon_{\Lambda}
    \quad\Longrightarrow\quad
    \lvert\Phi\rvert>\frac{m}{4}
    \quad\text{on }B_{\frac r4}(x).
\end{equation}
\end{corollary}
\begin{proof}
Let
\[
    R_{\Lambda}:=32C_g\Lambda
    \quad\text{and}\quad
    \varepsilon_{\Lambda}
    :=
    \varepsilon_{1/2,R_{\Lambda}}
    =
    \min\left\{
        8^{-1}C_{R_{\Lambda}}^{-1}R_{\Lambda},
        \varepsilon_0
    \right\}.
\]
By \eqref{eq:finite-energy-radius-estimate}, applied with
$\delta=1/2$, we have
\[
    r_{1/2}(x)
    \leqslant
    4C_g\Lambda m^{-1}
    <
    \frac{R_{\Lambda}}{4m}.
\]
Therefore
\[
    \sup_{\partial\overline B_{\frac r4}(x)}
    \lvert\Phi\rvert
    \geqslant
    \frac m2.
\]
The conclusion now follows from
Theorem~\ref{thm: zeros_in_sigma}, applied with
$\delta=1/2$ and $R=R_{\Lambda}$.
\end{proof}

\section{The blow-up set and the zero set}\label{sec: blow_up_and_zero}

From now on, we consider a sequence
$\lbrace(A_i,\Phi_i)\rbrace_{i\in\mathbb N}
\subseteq\mathcal A(E)\times\Gamma(\mathfrak g_E)$
of YMH critical points on $(X^3,g)$ satisfying the uniform bound
\begin{equation}\label{eq: reasonable?}
    m_i^{-1}\mathscr E_X(A_i,\Phi_i)\leqslant C,
\end{equation}
for some constant $C>0$, and whose masses satisfy
$\limsup_i m_i=\infty$. After passing to a subsequence, we may assume
that $m_i\uparrow\infty$. We note that in case the $(A_i,\Phi_i)$ are $\SU(2)$ monopoles of fixed charge $k\neq 0$, the energy formula \eqref{eq: Energy_Formula} guarantees an a priori uniform bound of the form \eqref{eq: reasonable?} with equality for $C=4\pi k$.

In order to study such a sequence of large mass YMH critical points,
consider the Radon measures
\begin{equation}\label{eq: seq_radon_measures}
\mu_i:=m_i^{-1}e(A_i,\Phi_i)\vol_g.
\end{equation}
By \eqref{eq: reasonable?}, the sequence has uniformly bounded mass.
After passing to a subsequence, still denoted by the same indices, we may
therefore assume that
\[
    \mu_i\rightharpoonup\mu
\]
for some Radon measure $\mu$ on $X$.

For a fixed bounded-geometry radius $\rho_0>0$, set
\begin{equation}\label{eq:general-YMH-sets}
\begin{split}
S_j
&:=
\bigcap_{0<r\leqslant\rho_0}
\left\{
x\in X:
\liminf_{i\to\infty}
m_i^{-1}\mathscr{E}_{B_r(x)}(A_i,\Phi_i)
\geq j^{-1}
\right\},
\\
S
&:=
\bigcup_{j\geq1}S_j,
\qquad
Z
:=
\bigcap_{n\geq1}
\overline{\bigcup_{i\geq n}\Phi_i^{-1}(0)}.
\end{split}
\end{equation}
Thus $S$ and $Z$ agree with the notation introduced in
\eqref{eq:main-concentration-zero-sets}.

For each $x\in X$, let
\[
\mathcal R_x
:=
\left\{
r\in(0,\rho_0]:
\mu(\partial B_r(x))>0
\right\}.
\]
The set $\mathcal R_x$ is at most countable, and for every
$r\in(0,\rho_0]\setminus\mathcal R_x$ one has
\[
    \mu(B_r(x))
    =
    \lim_{i\to\infty}\mu_i(B_r(x)).
\]

Our first result relates the blow-up set $S$ and the zero set $Z$
defined in \eqref{eq:general-YMH-sets}. In the following statement, we
use $\mathcal{H}^0$ to denote the counting measure.
\begin{theorem}\label{thm: main_1}
	Let $(X^3,g)$ be an AC oriented Riemannian $3$-manifold with one end, $E$ be a $G$-bundle over $X$, and $\lbrace (A_i,\Phi_i)\rbrace_{i\in\mathbb{N}}\subseteq\mathcal{A}(E)\times\Gamma (\mathfrak{g}_E)$ be a sequence of YMH critical points on $(X^3,g)$ satisfying the uniform bound \eqref{eq: reasonable?} and whose masses $m_i$ satisfy $\limsup_{i\to\infty} m_i=\infty$. Then, after passing to a subsequence for which $m_i\uparrow\infty$ and
$\mu_i\rightharpoonup\mu$, the following hold:
	\begin{itemize}
		\item[(i)] $\mathcal{H}^0({S}_j)\leqslant jC$, for all $j\in\mathbb{N}$; in particular, each ${S}_j$ is finite and ${S}$ is countable.
		\item[(ii)] The blow-up set contains the zero set:
		\[
		Z\subseteq{S}.
		\]
	\end{itemize}
\end{theorem}
\begin{proof}
(i) Given $0<r\leqslant\rho_0$, we can find a countable open covering $\{B_{5r_l}(x_l)\}$ of ${S}_j$ with $x_l\in{S}_j$, $10r_l<r$ and $B_{r_l}(x_l)$ pairwise disjoint. Then
\begin{eqnarray*}
	\sum_l (5r_l)^0 &\leqslant& j\sum_l\liminf_{i\to\infty} m_i^{-1}\mathscr{E}_{B_{r_l}(x_l)}(A_i,\Phi_i)\quad\text{($x_l\in{S}_j$)}\\
	&\leqslant& j\liminf_{i\to\infty} m_i^{-1}\sum_l\mathscr{E}_{B_{r_l}(x_l)}(A_i,\Phi_i)\quad\text{(by Fatou's lemma)}\\
	&\leqslant& j\liminf_{i\to\infty} m_i^{-1}\mathscr{E}_{X}(A_i,\Phi_i)\quad\text{(the $B_{r_l}(x_l)$'s are disjoint)}\\
	&\leqslant& jC.\quad\text{(by \eqref{eq: reasonable?})}
\end{eqnarray*}
Since this bound is uniform in $r\in (0,\rho_0]$, it follows that $\mathcal{H}^0({S}_j)\leqslant jC$.

(ii) We shall apply Corollary \ref{cor: zeros_in_sigma} with $\Lambda:=2C$. Let $x_0\in X\setminus S$. Choose $j\in\mathbb N$ so large that
$j^{-1}<\varepsilon_\Lambda$. Since $x_0\notin S_j$, there is
$r_0\in(0,\rho_0]$ such that
\[
\liminf_{i\to\infty}
m_i^{-1}\mathscr{E}_{B_{r_0}(x_0)}(A_i,\Phi_i)
<j^{-1}<\varepsilon_\Lambda.
\]
Shrinking $r_0$ if necessary, we may assume
$r_0\notin\mathcal R_{x_0}$; the preceding inequality is preserved by
monotonicity with respect to the radius. In particular, it follows that there is $i_0\in\mathbb{N}$ such that
\[
m_{i}^{-1}\mathscr{E}_{B_{r_0}(x_0)}(A_{i},\Phi_{i})<\varepsilon_{\Lambda},\quad\forall i\geq i_0.
\] Since $m_i\uparrow\infty$, by increasing $i_0$ if necessary we may also assume that
\[
m_i>m_0
\quad\text{and}\quad
r_i:=R_{\Lambda}m_i^{-1}<\frac{r_0}{2},\quad\forall i\geq i_0.
\] Hence, given any $x\in B_{r_0/2}(x_0)$, it follows that
\[
m_{i}^{-1}\mathscr{E}_{B_{r_i}(x)}(A_{i},\Phi_{i})<\varepsilon_{\Lambda},\quad\forall i\geq i_0,
\] so that applying Corollary \ref{cor: zeros_in_sigma} we get that
\[
	\lvert\Phi_i\rvert(x)>\frac{m_i}{4},\quad\forall i\geq i_0.
\] Therefore, 
\[
	\inf_{B_{\frac{r_0}{2}}(x_0)}\lvert\Phi_i\rvert\geq\frac{m_i}{4}>0,\quad\forall i\geq i_0.
\] In particular, it follows that $x_0\in X\setminus Z$. By the arbitrariness of $x_0\in X\setminus{S}$, this shows that $Z\subset{S}$.
\end{proof}

\section{Bubbling}
\label{sec: bubbling}

\noindent\textbf{Corrected and expanded in version~5.}
The fixed centre rescaling of the published version is replaced by moving
centre arguments. For rank one Yang--Mills--Higgs critical points, a
scale selection shows that the concentration radius is comparable to the
inverse mass. For monopoles, the energy and degree identities give a
complete finite cluster decomposition and show that no mass-renormalized energy remains in the intervening regions.
\medskip

\subsection{Rank one Yang--Mills--Higgs critical points}
\label{sec:general-YMH-bubbling}

We now give the revised proof of
Theorem~\ref{thm: Main_YMH}\textup{(a')} for
$G=\SU(2)$ or $G=\SO(3)$. The point is not to prescribe the centre and
the scale in advance. Instead, one first selects a ball carrying a fixed
amount of mass-renormalized energy and then proves that its radius is
comparable
to the inverse mass. We use the rescaled functional and
$\varepsilon$-notation employed in \cite{cheng2025su2}, specialized to
the zero-potential case.

Set
\[
    \eps_i:=m_i^{-1},
    \qquad
    \phi_i:=m_i^{-1}\Phi_i,
\]
and, for $U\subset X$, introduce the scaled functional
\begin{equation}\label{eq:scaled-general-YMH-functional}
    \mathcal Y_{\eps_i,U}(A_i,\phi_i)
    :=
    \frac12\int_U
    \left(
       \eps_i^2\abs{F_{A_i}}^2
       +
       \abs{\text{d}_{A_i}\phi_i}^2
    \right)\vol_g.
\end{equation}
Then
\begin{equation}\label{eq:scaled-general-YMH-energy}
    \eps_i^{-1}\mathcal Y_{\eps_i,U}(A_i,\phi_i)
    =
    m_i^{-1}\mathscr E_U(A_i,\Phi_i),
\end{equation}
and the critical-point equations become
\begin{equation}\label{eq:scaled-general-YMH-equations}
    \eps_i^2\text{d}_{A_i}^*F_{A_i}
    =
    [\text{d}_{A_i}\phi_i,\phi_i],
    \qquad
    \Delta_{A_i}\phi_i=0.
\end{equation}
The maximum principle gives $\abs{\phi_i}\leqslant1$.

If $p_i\in X$ and $t_i\downarrow0$, let
\[
    s_i(z):=\exp_{p_i}(t_i z),
    \qquad
    \widetilde g_i:=t_i^{-2}s_i^*g,
    \qquad
    \widetilde\eps_i:=\frac{\eps_i}{t_i},
\]
and set
\[
    (\widetilde A_i,\widetilde\phi_i):=s_i^*(A_i,\phi_i).
\]
The pair $(\widetilde A_i,\widetilde\phi_i)$ is a critical point of
$\mathcal Y_{\widetilde\eps_i}$ with respect to $\widetilde g_i$, and
\begin{equation}\label{eq:general-YMH-rescaling-identity}
    \widetilde\eps_i^{-1}
    \mathcal Y_{\widetilde\eps_i,B_R(0),\widetilde g_i}
    (\widetilde A_i,\widetilde\phi_i)
    =
    \eps_i^{-1}
    \mathcal Y_{\eps_i,B_{Rt_i}(p_i),g}(A_i,\phi_i).
\end{equation}

We require two consequences of the refined analytic estimates and the
corrected regularity proof uniform in the mass.  The second uses the rank
one fact that the centralizer of a nonzero Higgs field is one-dimensional.

\begin{lemma}[Coarse mass scale bound]
\label{lem:general-YMH-coarse-density}
There is $m_0>0$, depending only on the bounded geometry of $(X,g)$,
such that for every $\Lambda<\infty$ there is $C_\Lambda<\infty$ with
the following property.  Every YMH critical point of mass $m\geqslant m_0$ satisfying
\[
    m^{-1}\mathscr E_X(A,\Phi)\leqslant\Lambda
\]
also satisfies
\begin{equation}\label{eq:general-YMH-coarse-density}
    \sup_X e(A,\Phi)\leqslant C_\Lambda m^4.
\end{equation}
\end{lemma}

\begin{proof}
Rescale the metric by $g_m=m^2g$ and the Higgs field by
$\phi=m^{-1}\Phi$. The rescaled configuration has mass parameter $1$, in particular $|\phi|\leqslant 1$ on $X$, and has total energy bounded by $\Lambda$. We spell out the
point at which the total energy bound enters. On a unit ball $B_1(x_0)$
for the rescaled metric, maximize
\[
 \theta(y)=\left(\frac12-d(x_0,y)\right)^3e(y)
\]
on $\overline{B}_{1/2}(x_0)$, as in the proof of
Theorem~\ref{thm: e-regularity}. Let $y_*$ be a
maximum point, write $e_*:=e(y_*)$, and set
\[
 s_*:=\frac12\left(\frac12-d(x_0,y_*)\right).
\]
The refined coercive Bochner inequality and the
local mean-value inequality give
\[
 e_*\leqslant C\left(s^{-3}\Lambda+s^2(e_*+e_*^{3/2})\right),
 \qquad 0<s\leqslant\min\{s_*,r_1\}.
\]
If \(e_*\leqslant1\), take \(s=\min\{s_*,s_1\}\), with \(s_1\) fixed and
small.  If \(e_*>1\), take
\(s=\min\{s_*,c e_*^{-1/4}\}\).  Absorbing the last term gives either a
bound for the weighted maximum \(\theta\), or
\[
 e_*\leqslant C\Lambda e_*^{3/4},
\]
and hence \(e_*\leqslant C(1+\Lambda^4)\).  Applying this argument on unit
balls centred at arbitrary points gives
\(\sup e\leqslant C_\Lambda\) in the rescaled metric.  This is the
large energy version of the Heinz argument; no smallness of \(\Lambda\)
is used.  Since an energy density in dimension three scales by \(m^4\),
scaling back gives \eqref{eq:general-YMH-coarse-density}.
\end{proof}

\begin{lemma}[The small parameter limit in rank one]
\label{lem:general-YMH-zero-parameter}
Let \(G=\SU(2)\) or \(G=\SO(3)\).  Let
\(\widetilde g_i\to g_{\mathbb R^3}\) smoothly on compact subsets of
\(\mathbb R^3\), let \(\widetilde\eps_i\to0\), and let
\((\widetilde A_i,\widetilde\phi_i)\) be critical points of
\(\mathcal Y_{\widetilde\eps_i}\) on an exhaustion of \(\mathbb R^3\).
Assume that they are restrictions of mass one configurations on the
corresponding complete rescaled manifolds and that
\[
 \sup_i\widetilde\eps_i^{-1}
 \mathcal Y_{\widetilde\eps_i}
 (\widetilde A_i,\widetilde\phi_i)<\infty.
\]
There is a constant \(\eta_{\mathrm{ad}}>0\), depending only on the
uniform bounded-geometry data and the global mass-renormalized energy bound,
with the following property.  If
\[
    \sup_{y\in B_R(0)}
    \widetilde\eps_i^{-1}
    \mathcal Y_{\widetilde\eps_i,B_1(y),\widetilde g_i}
       (\widetilde A_i,\widetilde\phi_i)
    \leqslant\eta_{\mathrm{ad}}
\]
for every fixed \(R\) and all sufficiently large \(i\), then, after
passing to a subsequence, there are harmonic one-forms
\(h_F,h_\phi\) on \(\mathbb R^3\) such that
\begin{equation}\label{eq:general-YMH-harmonic-density-limit}
    \widetilde\eps_i^{-1}
    e_{\widetilde\eps_i}
       (\widetilde A_i,\widetilde\phi_i)
    \longrightarrow
    \frac12\left(\abs{h_F}^2+\abs{h_\phi}^2\right)
\end{equation}
smoothly on compact subsets, where
\[
    e_{\eps}(A,\phi)
    :=
    \frac12\left(\eps^2\abs{F_A}^2+
                         \abs{\text{d}_A\phi}^2\right).
\]
\end{lemma}

\begin{proof}
Put
\[
 B_i:=\widetilde\eps_i^{1/2}
       *_{\widetilde g_i}F_{\widetilde A_i},
 \qquad
 C_i:=\widetilde\eps_i^{-1/2}
       d_{\widetilde A_i}\widetilde\phi_i.
\]
Then
\[
 \widetilde\eps_i^{-1}e_{\widetilde\eps_i}
 (\widetilde A_i,\widetilde\phi_i)
 =\frac12\left(|B_i|^2+|C_i|^2\right).
\]
We first establish both a pointwise estimate at the
\(\widetilde\eps_i\)-scale and a lower bound for the mass-normalized Higgs field. Decrease \(\eta_{\mathrm{ad}}\), if necessary, so that it is below
the threshold in Theorem~\ref{thm: e-regularity}. Fix compact sets
\(K\Subset K^+\Subset\mathbb R^3\) and a number \(R_0>4\). Apply that
theorem to the ordinary YMH configuration
\[
 \left(
    \widetilde A_i,
    \widetilde\eps_i^{-1}\widetilde\phi_i
 \right)
\]
on balls of radius \(R_0\widetilde\eps_i\), with the fixed parameter
\(R=R_0\). Since these balls are contained in the corresponding unit
balls for all large \(i\), the local hypothesis gives
\[
 \widetilde\eps_i
 \mathscr E_{B_{R_0\widetilde\eps_i}(y)}
 \left(
    \widetilde A_i,
    \widetilde\eps_i^{-1}\widetilde\phi_i
 \right)
 =
 \widetilde\eps_i^{-1}
 \mathcal Y_{\widetilde\eps_i,
 B_{R_0\widetilde\eps_i}(y)}
 (\widetilde A_i,\widetilde\phi_i)
 \leqslant\eta_{\mathrm{ad}}.
\]
Consequently, uniformly on \(K^+\),
\begin{equation}\label{eq:rank-one-relative-density-smallness}
 \widetilde\eps_i^4
 \left(
   |F_{\widetilde A_i}|^2+
   \widetilde\eps_i^{-2}
   |d_{\widetilde A_i}\widetilde\phi_i|^2
 \right)
 \leqslant C\eta_{\mathrm{ad}}.
\end{equation}

Let \(G_i\) be the minimal positive Green function on the corresponding
complete rescaled manifold and set
\[
 d\nu_i
 :=
 \widetilde\eps_i^{-1}
 |d_{\widetilde A_i}\widetilde\phi_i|^2
 \,\vol_{\widetilde g_i}.
\]
The Green representation gives
\[
 1-|\widetilde\phi_i(y)|^2
 =2\widetilde\eps_i\int G_i(y,z)\,d\nu_i(z).
\]
For \(y\in K\), the contribution of
\(B_{\widetilde\eps_i}(y)\) is at most
\(C\eta_{\mathrm{ad}}\) by
\eqref{eq:rank-one-relative-density-smallness} and the local bound
\(G_i(y,z)\leqslant C d_{\widetilde g_i}(y,z)^{-1}\). On each dyadic
annulus with radii between \(2^j\widetilde\eps_i\) and
\(2^{j+1}\widetilde\eps_i\), up to radius one, one has
\(\widetilde\eps_iG_i\leqslant C2^{-j}\), while
\(\nu_i(B_1(y))\leqslant2\eta_{\mathrm{ad}}\). Summing the annuli gives
another contribution bounded by \(C\eta_{\mathrm{ad}}\). Finally, the
complement of \(B_1(y)\) contributes \(o(1)\), by the global mass-renormalized energy bound. Thus
\[
 \sup_K\left(1-|\widetilde\phi_i|^2\right)
 \leqslant C\eta_{\mathrm{ad}}+o(1).
\]
Choosing \(\eta_{\mathrm{ad}}\) sufficiently small, we obtain, for all
large \(i\),
\begin{equation}\label{eq:rank-one-Higgs-lower-bound}
    |\widetilde\phi_i|\geqslant\frac12
    \qquad\text{on }K.
\end{equation}

Write
\[
    n_i:=\frac{\widetilde\phi_i}{|\widetilde\phi_i|}.
\]
For \(G=\SU(2)\) or \(G=\SO(3)\), the orthogonal complement of
\(\mathbb R n_i\) is the image of \(\operatorname{ad}n_i\), and there is
a uniform constant \(c>0\) such that
\begin{equation}\label{eq:rank-one-commutator-coercivity}
    |[\widetilde\phi_i,\xi]|
    \geqslant c|\xi^\perp|
\end{equation}
whenever \(|\widetilde\phi_i|\geqslant1/2\).  Apply
Lemma~\ref{lem: Bochner_estimate} to the ordinary
YMH configuration
\[
    \left(
       \widetilde A_i,
       \widetilde\eps_i^{-1}\widetilde\phi_i
    \right),
\]
whose mass is \(\widetilde\eps_i^{-1}\). Fix relatively compact smooth
domains
\[
    U\Subset U'\Subset U''\Subset\mathbb R^3.
\]
Applying \eqref{eq:rank-one-relative-density-smallness} with a compact
set containing \(U''\), and decreasing \(\eta_{\mathrm{ad}}\) if
necessary, the
curvature terms in the local Bochner inequalities for
\[
 \left[
 F_{\widetilde A_i},
 \widetilde\eps_i^{-1}\widetilde\phi_i
 \right],
 \qquad
 \left[
 d_{\widetilde A_i}
   (\widetilde\eps_i^{-1}\widetilde\phi_i),
 \widetilde\eps_i^{-1}\widetilde\phi_i
 \right]
\]
are absorbed by the coercive zeroth-order terms furnished by
\eqref{eq:rank-one-Higgs-lower-bound} and
\eqref{eq:rank-one-commutator-coercivity}.  The resulting inequalities
have the form
\[
    \Delta |\Xi_i|^2
    \leqslant
    -c\widetilde\eps_i^{-2}|\Xi_i|^2
    \qquad\text{on }U'',
\]
for the two transverse commutators \(\Xi_i\).  The coarse density bound
gives a polynomial bound on \(\partial U''\).  Comparison with
\(\exp(-c\,d(\,\cdot\,,\partial U'')/\widetilde\eps_i)\), followed by
interior estimates for the differentiated equations, therefore gives,
for every \(k\geqslant0\),
\begin{equation}\label{eq:rank-one-adiabatic-estimates}
 \widetilde\eps_i^{-1}
 \left(
 \|[B_i,\widetilde\phi_i]\|_{C^k(U')}
 +\|[C_i,\widetilde\phi_i]\|_{C^k(U')}
 \right)
 \longrightarrow0.
\end{equation}
This is the local version of the large-Higgs estimates in
\cite{fadel2023asymptotics}*{Lemma~4.8 and Theorem~4.11}; the argument is
pointwise and uses only bounded geometry, the lower bound for the Higgs
field, and \eqref{eq:rank-one-relative-density-smallness}.

Decompose
\[
    B_i=b_i n_i+B_i^\perp,
    \qquad
    C_i=c_i n_i+C_i^\perp,
\]
where \(b_i\) and \(c_i\) are ordinary real one-forms.  By
\eqref{eq:rank-one-commutator-coercivity} and
\eqref{eq:rank-one-adiabatic-estimates}, the transverse components
\(B_i^\perp\) and \(C_i^\perp\) tend to zero smoothly on compact
subsets.  Moreover,
\[
    d_{\widetilde A_i}n_i
    =|\widetilde\phi_i|^{-1}
      \bigl(d_{\widetilde A_i}\widetilde\phi_i\bigr)^\perp
    \longrightarrow0
\]
smoothly locally.  The Bianchi identity and the critical-point equations
give
\[
 d_{\widetilde A_i}^{*}B_i=0,
 \qquad
 d_{\widetilde A_i}^{*}C_i=0,
\]
and express \(d_{\widetilde A_i}B_i\) and
\(d_{\widetilde A_i}C_i\) as fixed Hodge-star multiples of
\[
 \widetilde\eps_i^{-1}[C_i,\widetilde\phi_i]
 \quad\text{and}\quad
 \widetilde\eps_i^{-1}[B_i,\widetilde\phi_i],
\]
respectively.  Projecting onto \(n_i\), and using
\eqref{eq:rank-one-adiabatic-estimates} together with
\(d_{\widetilde A_i}n_i\to0\), gives
\[
    db_i\to0,
    \qquad d^*b_i\to0,
    \qquad dc_i\to0,
    \qquad d^*c_i\to0
\]
smoothly on compact subsets.

The global mass-renormalized energy bound gives uniform local \(L^2\) bounds
for \(b_i\) and \(c_i\).  Interior estimates for the Hodge operator
\(d+d^*\) therefore yield uniform \(W^{1,2}\) bounds on smaller compact
sets.  Applying the same estimates after differentiating the preceding
equations gives uniform bounds in every Sobolev norm.  Hence, after
passing to a diagonal subsequence, \(b_i\to h_F\) and
\(c_i\to h_\phi\) smoothly locally, where \(h_F\) and \(h_\phi\) are
closed and coclosed one-forms.  Since the transverse components vanish,
the convergence of the energy densities is precisely
\eqref{eq:general-YMH-harmonic-density-limit}.
\end{proof}

\begin{proposition}[Selection of the moving centre and the mass scale]
\label{prop:general-YMH-moving-centre}
Fix $x\in S$. After passing to a subsequence, there are $p_i\to x$,
$t_i\downarrow0$, and constants $0<c_x\leqslant C_x<\infty$ such that
\begin{equation}\label{eq:general-YMH-comparable-scales}
    c_x\leqslant m_it_i\leqslant C_x.
\end{equation}
Moreover, after changing gauge, a subsequence of
$(\widetilde A_i,\widetilde\phi_i)$ converges smoothly on compact
subsets of $\mathbb R^3$ to a nontrivial critical point
$(A_x,\phi_x)$ of $\mathcal Y_\alpha$, where
\[
    \alpha
    =
    \lim_{i\to\infty}\frac{\eps_i}{t_i}
    \in(0,\infty).
\]
The pair $(A_x,\alpha^{-1}\phi_x)$ is a Euclidean YMH critical point of mass $\alpha^{-1}$. Applying the standard Euclidean
Yang--Mills--Higgs scaling then yields a nontrivial mass one Euclidean YMH critical point.
\end{proposition}

\begin{proof}
Choose $j$ with $x\in S_j$, and fix
\[
    0<\eta<
    \min\left\{\frac1{4j},\eta_{\mathrm{ad}}\right\}.
\]
Choose numbers \(R_i\uparrow\infty\). We use the concentration-radius
and moving centre selection underlying the bubbling analysis in
\cite{cheng2025su2}*{Sections~5.5--5.7}, with the modifications required
by the absence of the Higgs potential. Applied on a diagonal sequence of
nested balls about \(x\), it gives radii \(\rho_i\downarrow0\), centres
\(p_i\to x\), and scales \(t_i\downarrow0\) such that
\begin{equation}\label{eq:general-YMH-selected-energy}
    \eps_i^{-1}
    \mathcal Y_{\eps_i,B_{t_i}(p_i)}(A_i,\phi_i)
    =
    \eta,
\end{equation}
\[
    \frac{\rho_i}{t_i}\geqslant R_i\longrightarrow\infty,
\]
and, for every fixed \(R<\infty\),
\begin{equation}\label{eq:general-YMH-selected-smallness}
 \sup_{y\in B_R(0)}
 \eps_i^{-1}
 \mathcal Y_{\eps_i,
 B_{t_i}(\exp_{p_i}(t_i y))}(A_i,\phi_i)
 \leqslant\eta
\end{equation}
for all sufficiently large \(i\).  We briefly recall the selection.  By
the definition of \(S_j\), one first chooses \(\rho_i\) so that a ball
of radius \(\rho_i/R_i\) about \(x\) carries mass-renormalized energy at least
\(2\eta\).  One then takes the smallest radius at which a ball in a
slightly larger nested neighbourhood carries energy \(\eta\), and
chooses its centre.  Repeating the choice with an increasing number of
nested neighbourhoods and passing to a diagonal subsequence gives
\eqref{eq:general-YMH-selected-smallness}; the unused neighbourhood
between successive balls provides the required boundary margin.  This is
the standard moving centre selection in the bubbling argument.  In
particular, the \(t_i\)-rescaled domains exhaust \(\mathbb R^3\).

By Lemma~\ref{lem:general-YMH-coarse-density},
\[
    \eta
    \leqslant
    C_\Lambda m_i^3t_i^3,
\]
and hence $m_it_i\geqslant c_x>0$.

Suppose, towards a contradiction, that $m_it_i\to\infty$ along a
subsequence. Then $\widetilde\eps_i=\eps_i/t_i\to0$. Rescale by $t_i$ as
in \eqref{eq:general-YMH-rescaling-identity}.  The bound
\eqref{eq:general-YMH-selected-smallness} gives the small energy hypothesis of
Lemma~\ref{lem:general-YMH-zero-parameter} on every fixed ball, with
equality on $B_1(0)$.  These rescalings have mass one, and their global
mass-renormalized energies equal
\(m_i^{-1}\mathscr E_X(A_i,\Phi_i)\), so the remaining hypotheses of
that lemma also hold.  We consequently obtain harmonic real $1$-forms
\(h_F,h_\phi\) on $\mathbb R^3$ satisfying
\[
    \frac12\int_{B_1(0)}
    \left(\abs{h_F}^2+\abs{h_\phi}^2\right)=\eta>0.
\]
On the other hand, the global mass-renormalized energy bound and
\eqref{eq:general-YMH-rescaling-identity} give
\[
    \frac12\int_{\mathbb R^3}
    \left(\abs{h_F}^2+\abs{h_\phi}^2\right)
    \leqslant
    \liminf_{i\to\infty}
    m_i^{-1}\mathscr E_X(A_i,\Phi_i)
    \leqslant C.
\]
The forms \(h_F\) and \(h_\phi\) are therefore $L^2$ harmonic
$1$-forms on $\mathbb R^3$, and hence vanish. This contradicts the positive energy on $B_1(0)$. Therefore
$m_it_i\leqslant C_x$, proving
\eqref{eq:general-YMH-comparable-scales}.

After extracting once more, $\widetilde\eps_i\to\alpha\in(0,\infty)$.
The corrected small energy estimate, elliptic regularity in local Coulomb
gauges, and \eqref{eq:general-YMH-selected-smallness} give smooth convergence on
compact subsets to a critical point
$(A_x,\phi_x)$ of $\mathcal Y_\alpha$. Equation
\eqref{eq:general-YMH-selected-energy} passes to the limit and yields
\[
    \alpha^{-1}
    \mathcal Y_{\alpha,B_1(0)}(A_x,\phi_x)
    =
    \eta,
\]
so the limit is nontrivial. The global mass-renormalized energy bound gives
finite total action.

It remains to identify its mass. Fix $\delta\in(0,1)$. Since
\[
    m_i^{-1}\|\textnormal{d}_{A_i}\Phi_i\|_{L^2(X)}^2
    \leqslant
    2m_i^{-1}\mathscr E_X(A_i,\Phi_i)
    \leqslant
    2C,
\]
the finite energy radius estimate
\eqref{eq:finite-energy-radius-estimate}, applied with centre $p_i$,
gives a constant $R_\delta<\infty$ (e.g. take $R_\delta>2C_gC(1-\delta)^{-2}$), independent of $i$, such that
\[
    \sup_{\partial B_{R_\delta m_i^{-1}}(p_i)}
       \abs{\Phi_i}
    \geqslant
    \delta m_i.
\]
In the $t_i$-rescaled coordinates this sphere has radius
$R_\delta/(m_it_i)$, which stays in a fixed compact annulus by
\eqref{eq:general-YMH-comparable-scales}. After passing to a further
subsequence, the corresponding radii converge to a positive finite radius.
Smooth convergence therefore gives a point at which
$\abs{\phi_x}\geqslant\delta$.  The pair
$(A_x,\alpha^{-1}\phi_x)$ is an ordinary
YMH critical point on $\mathbb R^3$. By
\cite{fadel2023asymptotics}*{Corollary~4.5}, the pair has a
well-defined mass $\alpha^{-1}\overline m_x\geqslant0$ and
$\abs{\phi_x}\to\overline m_x$ uniformly at infinity.  Since
$\abs{\phi_x}$ is subharmonic, its supremum is at most
$\overline m_x$, and hence $\overline m_x\geqslant\delta$.  The bound
$\abs{\phi_i}\leqslant1$ passes to the limit, so
$\overline m_x\leqslant1$.  Letting $\delta\uparrow1$ gives
$\overline m_x=1$.  Thus the pair
$(A_x,\alpha^{-1}\phi_x)$ has mass $\alpha^{-1}$, and the standard
Euclidean Yang--Mills--Higgs scaling normalizes its mass to one,
completing the proof.
\end{proof}

\medskip
\noindent\textbf{Additional observation in version~5.}
\begin{remark}[Charged bubbles and Higgs zeros]
\label{rem:charged-bubbles-higgs-zeros}
Assume that $G=\SU(2)$ or $G=\SO(3)$, let $x\in S$, and use the
notation of the proof of Proposition~\ref{prop:general-YMH-moving-centre}.
Suppose that the resulting Euclidean Yang--Mills--Higgs bubble has
nonzero magnetic charge. Then the centres may be replaced by zeros of
the Higgs fields. More precisely, there are points
\[
    z_i\in\Phi_i^{-1}(0),
    \qquad
    z_i\longrightarrow x,
    \qquad
    d(z_i,p_i)=O(t_i),
\]
such that recentering at $z_i$ and rescaling by $t_i$ produces a
translate of the same nontrivial Euclidean bubble, with its Higgs field
vanishing at the origin.

Indeed, before the final fixed Euclidean normalization, write the smooth
limit obtained in the proof as $(A_x,\phi_x)$. Nonzero magnetic charge
means that, for every sufficiently large $R$, the normalized Higgs field
on $\partial B_R(0)$ represents a nontrivial class in
$\pi_2(S^2)$. Equivalently, its degree is nonzero; with the usual
normalization for $\SO(3)$, the numerical magnetic charge may differ
from this degree by a factor of $1/2$. In the $t_i$-rescaled normal
coordinates centred at $p_i$, set
\[
    \widehat\phi_i(y)
    :=m_i^{-1}\Phi_i\bigl(\exp_{p_i}(t_i y)\bigr).
\]
After choosing the gauges used in the bubbling argument,
$\widehat\phi_i\to\phi_x$ smoothly on $\overline{B_R(0)}$. Hence, for
all sufficiently large $i$, $\widehat\phi_i$ is nonzero on
$\partial B_R(0)$ and its normalization there has the same nonzero
topological degree as $\phi_x/\abs{\phi_x}$. If
$\widehat\phi_i$ were nonvanishing on $B_R(0)$, this normalized boundary
map would extend over the ball, which is impossible. Thus there is
$y_i\in B_R(0)$ such that $\widehat\phi_i(y_i)=0$. Setting
\[
    z_i:=\exp_{p_i}(t_i y_i)
\]
gives the asserted zeros. After passing to a subsequence, $y_i$ converges,
and recentering at $z_i$ therefore translates the limiting bubble so that
its Higgs field vanishes at the origin.

Consequently, if $S_{\mathrm{ch}}\subset S$ denotes the set of points
which admit a bubble of nonzero magnetic charge, then
\[
    S_{\mathrm{ch}}\subset Z.
\]
This argument does not apply to zero-charge bubbles, and therefore does
not by itself prove the reverse inclusion $S\subset Z$ for general rank
one Yang--Mills--Higgs critical points.
\end{remark}

\begin{proof}[Completion of the proof of Theorem~\ref{thm: Main_YMH}]
Part~\textup{(a')} is Proposition~\ref{prop:general-YMH-moving-centre}.
Parts~\textup{(b')} and~\textup{(c')} were proved in
Theorem~\ref{thm: main_1}.
\end{proof}

\begin{remark}[Why the rank one restriction is needed]
\label{rem:higher-rank-obstruction}
In higher rank, the scale of a nonabelian component need not be determined
by the norm of the full Higgs field.  To see this, equip
\(\mathfrak{su}(3)\) with the invariant inner product
\(\langle a,b\rangle=-2\operatorname{tr}(ab)\), whose restriction to
the upper-block copy of \(\mathfrak{su}(2)\) is the normalization used
above.  Set
\[
    H:=\frac{1}{\sqrt{12}}\operatorname{diag}(i,i,-2i)
    \in\mathfrak{su}(3).
\]
Then \(H\) is a unit element, orthogonal to the upper-block copy of
\(\mathfrak{su}(2)\), and commutes with it.

Let \(q_i\in\mathbb N\) satisfy \(q_i\to\infty\).  Using the classical
construction of well-separated Euclidean multimonopoles, choose mass one
\(\SU(2)\) monopoles \((A_i^0,\Psi_i^0)\) of charge \(q_i\).  The
constituent separation may be chosen, depending on \(q_i\), so large that
\begin{equation}\label{eq:higher-rank-uniform-density}
    \sup_{\mathbb R^3}e(A_i^0,\Psi_i^0)\leqslant C
\end{equation}
for a constant independent of \(i\).  Indeed, in the well-separated
construction the solution is arbitrarily close, near each constituent,
to a charge one BPS monopole and has arbitrarily small energy density
outside the constituent balls.  By increasing the separation as \(q_i\)
grows, the errors can be made uniform; see \cite{Jaffe1980}.  Choose
\(R_i\geqslant1\) such that
\[
    \mathscr E_{\mathbb R^3\setminus B_{R_i}}
       (A_i^0,\Psi_i^0)\leqslant1,
\]
and set
\[
    \widehat m_i:=q_iR_i.
\]
After applying the standard monopole scaling, we obtain an embedded
\(\SU(2)\) monopole \((A_i,\Psi_i)\) of mass \(\widehat m_i\) and charge
\(q_i\).  Through the upper-block inclusion
\(\SU(2)\hookrightarrow\SU(3)\), extend \((A_i,\Psi_i)\) to the
trivial \(\SU(3)\)-bundle over \(\mathbb R^3\), retaining the same
notation.  Thus \(A_i\) and \(\Psi_i\) take values in the embedded
copy of \(\mathfrak{su}(2)\).  Define
\[
    \Phi_i:=\Psi_i+\widehat m_iq_iH.
\]

Since \(H\) commutes with the embedded \(\mathfrak{su}(2)\), it is
parallel with respect to \(A_i\).  Therefore
\[
    \text{d}_{A_i}\Phi_i=\text{d}_{A_i}\Psi_i,
\]
and hence \((A_i,\Phi_i)\) is again a monopole.  Moreover, \(H\) is
orthogonal to \(\Psi_i\), so
\[
    |\Phi_i|^2
    =
    |\Psi_i|^2+\widehat m_i^2q_i^2.
\]
It follows that the mass of the \(\SU(3)\) configuration is
\[
    m_i
    =
    \widehat m_i\sqrt{1+q_i^2}.
\]
Its energy is unchanged by the addition of the parallel Higgs component:
\begin{equation}\label{eq:higher-rank-example-energy}
\begin{split}
    \mathscr E_{\mathbb R^3}(A_i,\Phi_i)
    &=
    \mathscr E_{\mathbb R^3}(A_i,\Psi_i)
    \\
    &=
    4\pi\widehat m_iq_i
    =
    4\pi m_i\frac{q_i}{\sqrt{1+q_i^2}}.
\end{split}
\end{equation}
In particular,
\[
    m_i^{-1}\mathscr E_{\mathbb R^3}(A_i,\Phi_i)
    \longrightarrow4\pi.
\]
Here \(q_i\) denotes the charge of the embedded \(\SU(2)\) monopole; no
charge for the full \(\SU(3)\) configuration is being used.

Under the scaling from mass one to mass \(\widehat m_i\),
\[
    \mathscr E_{\mathbb R^3\setminus B_{R_i/\widehat m_i}}
       (A_i,\Psi_i)
    =
    \widehat m_i
    \mathscr E_{\mathbb R^3\setminus B_{R_i}}
       (A_i^0,\Psi_i^0)
    \leqslant\widehat m_i.
\]
Since
\[
    \frac{\widehat m_i}{m_i}
    =
    \frac{1}{\sqrt{1+q_i^2}}
    \longrightarrow0
\]
and
\[
    \frac{R_i}{\widehat m_i}
    =
    \frac1{q_i}
    \longrightarrow0,
\]
we obtain
\[
    m_i^{-1}e(A_i,\Phi_i)\,\vol_{g_E}
    \rightharpoonup
    4\pi\delta_0.
\]
Thus the mass-renormalized energy concentrates at the origin.

On the other hand, the nonabelian component varies at the scale
\(\widehat m_i^{-1}\), whereas
\[
    \frac{m_i}{\widehat m_i}
    =
    \sqrt{1+q_i^2}
    \longrightarrow\infty.
\]
Indeed, after rescaling by \(m_i\), the bound
\eqref{eq:higher-rank-uniform-density} and the monopole scaling give
\[
    e(\widehat A_i,\widehat\Phi_i)
    \leqslant
    C\left(\frac{\widehat m_i}{m_i}\right)^4
    =
    \frac{C}{(1+q_i^2)^2}
    \longrightarrow0
\]
locally uniformly, independently of the choice of moving centres.
Furthermore,
\[
    \frac{\widehat m_iq_i}{m_i}
    =
    \frac{q_i}{\sqrt{1+q_i^2}}
    \longrightarrow1,
    \qquad
    \frac{\widehat m_i}{m_i}
    \longrightarrow0.
\]
Consequently, every pointed limit at a scale comparable to \(m_i^{-1}\)
is gauge equivalent to the flat connection with constant Higgs field
\(H\).  In particular, no nontrivial mass one Euclidean
Yang--Mills--Higgs configuration is obtained at the scale determined by
the total mass.

The obstruction is that the mass-normalized asymptotic Higgs fields converge to \(H\), which lies on a Weyl wall: its centralizer contains the
embedded \(\mathfrak{su}(2)\).  The coercivity mechanism used in the
rank one argument above is restored, for example, under the uniform
regularity condition
\[
    \inf_i\
    \inf_{\substack{
        \xi\perp\ker\operatorname{ad}\Phi_{\infty,i}\\
        |\xi|=1}}
    \frac{|[\Phi_{\infty,i},\xi]|}{m_i}
    >0.
\]
For \(G=\SU(2)\) or \(G=\SO(3)\), this condition is automatic, since the
centralizer of every nonzero element is one-dimensional.
\end{remark}

\subsection{Monopoles: concentration measures and complete bubble clusters}

For the remainder of this section, let $E\to X$ be an $\SU(2)$-bundle
and let $(A_i,\Phi_i)$ be finite energy monopoles of fixed charge $k>0$
and masses $m_i\to\infty$. This subsection replaces the fixed centre
bubbling and measure arguments in the published version. It proves the
corrected coefficient in Theorem~\ref{thm: Main_Monopoles}\textup{(d)}
and, as an additional result, gives a complete description of the finite
cluster of mutually separating mass one Euclidean monopoles over each
concentration point. There are no nontrivial profiles below the mass scale
and no secondary concentration inside an extracted mass one monopole.

Set
\[
    e_i:=\abs{\text{d}_{A_i}\Phi_i}^2=\abs{F_{A_i}}^2,
    \qquad \mu_i:=m_i^{-1}e_i\vol_g.
\]
The energy identity gives
\begin{equation}\label{eq:total-mass-mu}
    \mu_i(X)=4\pi k.
\end{equation}
Here and below we use the energy identity from
\eqref{eq: Energy_Formula}.

Introduce the mass-normalized fields
\begin{equation}\label{eq:mass-normalized-monopoles}
    \eps_i:=m_i^{-1},
    \qquad
    \phi_i:=m_i^{-1}\Phi_i.
\end{equation}
They satisfy
\begin{equation}\label{eq:scaled-bogomolny}
    \eps_i*F_{A_i}=\text{d}_{A_i}\phi_i,
    \qquad
    \abs{\phi_i}\leqslant1,
\end{equation}
where the pointwise bound follows from the maximum principle and the
finite energy asymptotic formula $|\Phi_i|\to m_i$.
Moreover, on every Borel set $U\subset X$,
\begin{equation}\label{eq:mass-renormalized-functional-measure}
\begin{split}
    \eps_i^{-1}\int_U
    \left(\eps_i^2\abs{F_{A_i}}^2
          +\abs{\text{d}_{A_i}\phi_i}^2\right)\vol_g
    &=2\mu_i(U),\\
    2\left\langle F_{A_i}\wedge \text{d}_{A_i}\phi_i\right\rangle
    &=2\mu_i.
\end{split}
\end{equation}
The second identity is an equality of Radon measures, with the
$3$-form on the left interpreted using the orientation of $X$. Thus energy and charge do
not separate for monopoles.

We first record the small energy description away from concentration.
The subsequent extraction uses only this local regularity, the Bogomolny
scaling law, and the degree of the normalized Higgs field.

\begin{lemma}[Compactness at the mass scale near a Higgs zero]
\label{lem:zero-centred-mass-scale}
Let \(p_i\) remain in a fixed compact subset of \(X\), and suppose that
\(\Phi_i(p_i)=0\).  Put
\[
 s_{i,p_i}(z)=\exp_{p_i}(\eps_i z),\qquad
 g_{i,p_i}=\eps_i^{-2}s_{i,p_i}^*g,
\]
and consider the rescalings at the mass scale
\[
 (A_i',\Phi_i')=s_{i,p_i}^*(A_i,\phi_i).
\]
After passing to a subsequence and changing gauge, these converge
smoothly on compact subsets of \(\mathbb R^3\) to a mass one Euclidean
monopole \((A_\infty,\Phi_\infty)\) of charge
\(q\in\{1,\ldots,k\}\).  Moreover,
\[
 \lim_{R\to\infty}\lim_{i\to\infty}
 \mu_i(B_{R\eps_i}(p_i))
 =\mathscr E_{\mathbb R^3}(A_\infty,\Phi_\infty)=4\pi q.
\]
\end{lemma}

\begin{proof}
The metrics \(g_{i,p_i}\) converge smoothly on compact subsets to the
Euclidean metric.  Lemma~\ref{lem:general-YMH-coarse-density}, applied
with \(\Lambda=4\pi k\), gives a uniform pointwise bound for the
mass one rescaled energy density.  The Bogomolny equation, the Bianchi
identity, and Coulomb gauge fixing therefore give uniform bounds for all
derivatives on smaller balls.  More explicitly, the curvature is
uniformly bounded in \(L^p\), \(p>3/2\), on each fixed ball; hence
\cite{uhlenbeck1982connections}*{Theorem~1.3} supplies Coulomb gauges,
and the first-order Bogomolny equation together with
\(\text{d}_A^*F_A=[\text{d}_A\Phi,\Phi]\) and
\(\Delta_A\Phi=0\) supplies the elliptic bootstrap.  A diagonal argument gives a smooth Euclidean monopole limit with
$\Phi_\infty(0)=0$ and energy at most $4\pi k$. We next show that its
mass is one; this will in particular imply that the limit is
nontrivial.

It remains to identify its mass.  The Green representation
\cite{fadel2023asymptotics}*{Theorem~3.11}, divided by \(m_i^2\), is
\begin{equation}\label{eq:green-mass-scale-zero}
 1-|\Phi_i'(z)|^2
 =
 2\int_X\eps_iG(s_{i,p_i}(z),y)\,d\mu_i(y).
\end{equation}
On fixed rescaled compact sets the kernel converges, away from the
diagonal, to the Euclidean Green kernel.  The global bound
\(\mu_i(X)=4\pi k\) and the standard AC Green estimate
\(G(x,y)\leqslant C\,d(x,y)^{-1}\) imply that the contribution of
\(d(y,p_i)\geq L\eps_i\) is at most \(Ck/L+o(1)\), locally uniformly in
\(z\).  The estimate follows from
\cite{fadel2023asymptotics}*{Corollary~2.9 and equation~(2.13)}; on the
fixed compact coordinate ball it is also the usual local Green bound, and
outside that ball the extra factor \(\eps_i\) makes the contribution
\(o(1)\).  Passing first to \(i\to\infty\) and then to \(L\to\infty\)
gives
\[
 1-|\Phi_\infty(z)|^2
 =
 2\int_{\mathbb R^3}G_{\mathbb R^3}(z,w)
 |\text{d}_{A_\infty}\Phi_\infty(w)|^2\,dw.
\]
The right-hand side tends to zero at infinity by
\cite{fadel2023asymptotics}*{Theorem~3.11}, now applied on
\(\mathbb R^3\).  Thus the limiting mass is one.

The integrality and energy formula in
\cite{fadel2023asymptotics}*{Theorem~1.1} give
\(\mathscr E_{\mathbb R^3}=4\pi q\) with \(q\in\mathbb N\).  Lower
semicontinuity and the global energy bound give \(q\leqslant k\).
For each fixed \(R\), smooth convergence and the scaling identity give
\[
 \lim_{i\to\infty}\mu_i(B_{R\eps_i}(p_i))
 =\mathscr E_{B_R(0)}(A_\infty,\Phi_\infty).
\]
Letting \(R\to\infty\) proves the final assertion.
\end{proof}

\begin{lemma}[Local abelianization on a zero-free region]
\label{lem:zero-free-abelianization}
Let \(B_{2r}(x)\) be a fixed geodesic ball.  Suppose that, for all
sufficiently large \(i\),
\[
 \Phi_i^{-1}(0)\cap B_{2r}(x)=\varnothing,\qquad
 \inf_{B_{2r}(x)}|\phi_i|\longrightarrow1,
\]
and that no mass scale concentration occurs there, in the sense that for
every fixed \(R<\infty\),
\[
 \sup_{p\in B_{3r/2}(x)}
 \mu_i(B_{R\eps_i}(p))\longrightarrow0.
\]
Then, after passing to a subsequence,
\[
 h_i:=m_i^{-1/2}d|\Phi_i|
\]
converges smoothly on \(B_r(x)\) to a harmonic one-form \(h\), and
\[
 m_i^{-1}|F_{A_i}|^2
 =m_i^{-1}|\text{d}_{A_i}\Phi_i|^2
 \longrightarrow |h|^2
\]
smoothly on \(B_r(x)\).
\end{lemma}

\begin{proof}
Write \(\Psi_i=\Phi_i/|\Phi_i|\) and use the orthogonal splitting
\[
 \mathfrak g_E
 =\langle\Psi_i\rangle\oplus\langle\Psi_i\rangle^\perp.
\]
The algebraic formulae for this splitting are given in
\cite{fadel2023asymptotics}*{Equations~(3.11)--(3.14)}. In particular,
\[
    |[\Phi_i,\xi]|
    \geq
    |\Phi_i|\,|\xi^\perp|.
\]

The nonconcentration hypothesis and
Theorem~\ref{thm: e-regularity}, applied on balls of
radius \(R\eps_i\) with \(R\) fixed, imply
\[
 m_i^{-4}\sup_{B_{4r/3}(x)}
 \left(|F_{A_i}|^2+|\text{d}_{A_i}\Phi_i|^2\right)\longrightarrow0.
\]
The local Bochner inequalities
\cite{fadel2023asymptotics}*{Lemma~4.8}, whose proof is pointwise and
does not use the geometry of the end, can therefore be applied with
\(|\Phi_i|\geq m_i/2\).  After the lower-order terms are absorbed, they
give on \(B_{4r/3}(x)\)
\[
 \Delta |\Xi_i|^2\leqslant-cm_i^2|\Xi_i|^2,\qquad
 \Xi_i=\bigl([\text{d}_{A_i}\Phi_i,\Phi_i],
             [F_{A_i},\Phi_i]\bigr),
\]
for all large \(i\).  We spell out the local comparison used in
\cite{fadel2023asymptotics}*{Theorem~4.11}.  Smooth the distance
\(s(y)=d(y,\partial B_{4r/3}(x))\) on a fixed collar and compare
\(|\Xi_i|^2\) there with
\[
 C m_i^6\exp(-a m_i s(y)).
\]
The polynomial prefactor follows from the coarse density bound:
$|F_{A_i}|+|\text{d}_{A_i}\Phi_i|=O(m_i^2)$ and $|\Phi_i|=O(m_i)$ on the
fixed ball.  For
\(a>0\) small, bounded geometry gives
\((\Delta+c m_i^2)\exp(-a m_i s)\geq0\).  The maximum principle on the
collar and then on its interior gives exponential decay on \(B_r(x)\).
Interior estimates for the differentiated equations therefore give,
for every \(j\geq0\),
\begin{equation}\label{eq:local-transverse-decay}
 m_i^{-2-j}\left(
 |\nabla_{A_i}^j(F_{A_i})^\perp|
 +|\nabla_{A_i}^j(\text{d}_{A_i}\Phi_i)^\perp|
 \right)\longrightarrow0
\end{equation}
uniformly on \(B_r(x)\).  The polynomial prefactor supplied by the coarse
density bound is dominated by the exponential
\(\exp(-cm_ir)\).  Differentiated versions follow from the differentiated
Bogomolny equation and interior elliptic estimates.

Since
\[
 \text{d}_{A_i}\Phi_i
 =d|\Phi_i|\,\Psi_i+|\Phi_i|\text{d}_{A_i}\Psi_i,
\]
the second term is the transverse component.  Hence
\[
 m_i^{-1}|\text{d}_{A_i}\Phi_i|^2
 =|h_i|^2+o(1)
\]
smoothly on \(B_r(x)\).  The forms \(h_i\) are closed and have uniformly
bounded local \(L^2\)-norm.  The scalar identity following from
\(\Delta_{A_i}\Phi_i=0\) is
\[
 d^*h_i
 =-m_i^{-1/2}
 \frac{|(\text{d}_{A_i}\Phi_i)^\perp|^2}{|\Phi_i|}.
\]
By \eqref{eq:local-transverse-decay}, the right-hand side tends smoothly
to zero.  Interior estimates for \(d+d^*\), followed by
Rellich compactness and bootstrapping, give smooth convergence to a
closed and coclosed one-form \(h\).  Finally the Bogomolny equation gives
the same conclusion for \(m_i^{-1}|F_{A_i}|^2\).
\end{proof}

\begin{lemma}[Limit under a small energy hypothesis]
\label{lem:small-energy-limit}
There are $r_0,\eta_0>0$, depending only on the bounded geometry of $X$,
with the following property. Suppose that $B_{4r}(x)$ is a geodesic ball,
$0<r\leqslant r_0$, and
\[
    \limsup_{i\to\infty}\mu_i(B_{4r}(x))<\eta_0.
\]
Then $\Phi_i$ is zero-free on $B_{2r}(x)$ for all sufficiently large $i$.
After passing to a subsequence, the forms
\[
    h_i:=m_i^{-1/2}d\abs{\Phi_i}
\]
converge smoothly on $B_r(x)$ to a harmonic $1$-form $h$, while
\begin{equation}\label{eq:local-diffuse-limit}
    m_i^{-1}e_i\longrightarrow\abs{h}^2
    \qquad\text{smoothly on }B_r(x).
\end{equation}
The local limits obtained on overlapping small energy balls are compatible
on overlaps.
\end{lemma}

\begin{proof}
Choose \(\eta_0<4\pi\).  If zeros \(p_i\in B_{2r}(x)\) existed along a
subsequence, Lemma~\ref{lem:zero-centred-mass-scale} would give, for
every fixed \(R\),
\[
 \liminf_{i\to\infty}\mu_i(B_{R\eps_i}(p_i))
 =\mathscr E_{B_R(0)}(A_\infty,\Phi_\infty).
\]
Since \(B_{R\eps_i}(p_i)\subset B_{4r}(x)\) for large \(i\), letting
\(R\to\infty\) would give
\(\liminf_i\mu_i(B_{4r}(x))\geq4\pi\), a contradiction.  Thus the
fields are zero-free on \(B_{2r}(x)\).

We next prove that \(|\phi_i|\to1\) uniformly on \(B_{3r/2}(x)\).
Otherwise, after taking a subsequence, there are \(y_i\in B_{3r/2}(x)\)
and \(\delta>0\) with \(|\phi_i(y_i)|\leqslant1-\delta\). Set
\(c_\delta:=1-(1-\delta)^2>0\). The Green identity gives
\[
 c_\delta
 \leqslant
 1-|\phi_i(y_i)|^2
 =2\int_X\eps_iG(y_i,z)\,d\mu_i(z).
\]
Choose \(L\) so large that the contribution of
\(X\setminus B_{L\eps_i}(y_i)\) is at most \(c_\delta/4\), using the
Green estimate from \eqref{eq:green-mass-scale-zero}. By
Lemma~\ref{lem:general-YMH-coarse-density}, the density of \(\mu_i\) is
bounded by \(C\eps_i^{-3}\). Hence, for every fixed \(\sigma\in(0,1)\),
\[
 2\int_{B_{\sigma\eps_i}(y_i)}
   \eps_iG(y_i,z)\,d\mu_i(z)
 \leqslant C\sigma^2.
\]
Choose \(\sigma\) so small that this is at most \(c_\delta/4\). On the
remaining annulus
\(B_{L\eps_i}(y_i)\setminus B_{\sigma\eps_i}(y_i)\) one has
\(\eps_iG(y_i,z)\leqslant C/\sigma\). It follows that
\[
 \liminf_{i\to\infty}\mu_i(B_{L\eps_i}(y_i))
 \geqslant c(\delta,\sigma)>0.
\]
Rescaling about \(y_i\) at the mass scale, the coarse density bound,
Uhlenbeck gauge fixing, and the Bogomolny equation give smooth convergence
on compact subsets to a nontrivial finite energy Euclidean monopole. The
Green representation, exactly as in
Lemma~\ref{lem:zero-centred-mass-scale}, shows that its mass is one. Its
energy is positive, and hence its charge is a positive integer. The
normalized limiting Higgs field therefore has nonzero degree on every
sufficiently large regular sphere. By smooth convergence on such a
sphere, the normalized Higgs fields of the original sequence have the
same degree for all large \(i\), and consequently \(\Phi_i\) has a zero
in the enclosed mass-scale ball. This contradicts the zero-free property
of \(B_{2r}(x)\). Hence \(|\phi_i|\to1\).

The same argument shows that no mass scale concentration occurs on
\(B_{3r/2}(x)\): a fixed positive amount of energy on a ball of radius
\(R\eps_i\) would produce a positive charge mass one limit, whose
nonzero degree would force a nearby zero of \(\Phi_i\).
Lemma~\ref{lem:zero-free-abelianization} now applies and
gives the asserted smooth convergence and
\eqref{eq:local-diffuse-limit}.  Compatibility on overlaps follows
because \(h_i=m_i^{-1/2}d|\Phi_i|\) is globally defined wherever
\(\Phi_i\neq0\), so two subsequential limits agree after the common
diagonal extraction.
\end{proof}

By local weak-$*$ compactness of Radon measures, after passing to a
subsequence there is a Radon measure $\mu$ on $X$ such that
\[
    \mu_i\wto\mu.
\]
Fix $\eta_0$ as in Lemma~\ref{lem:small-energy-limit} and define
\begin{equation}\label{eq:threshold-concentration-set}
    T
    :=
    \left\{
       x\in X:\mu(\{x\})\geqslant\eta_0
    \right\}.
\end{equation}
Since $\mu(X)\leqslant4\pi k$, the set $T$ is finite. If
$x\notin T$, then for some sufficiently small $r>0$ one has
$\mu(\overline{B_{4r}(x)})<\eta_0$. The Portmanteau theorem and
Lemma~\ref{lem:small-energy-limit}, followed by a diagonal extraction over
a countable relatively compact cover of $X\setminus T$, therefore give a
harmonic $1$-form $h$ on $X\setminus T$ such that
\begin{equation}\label{eq:diffuse-part-measure}
    \mu|_{X\setminus T}=\abs{h}^2\vol_g.
\end{equation}
The bound $\mu_i(X)=4\pi k$ and local smooth convergence imply that
$h\in L^2(X\setminus T)$.

\begin{lemma}[Absence of a diffuse limit of the mass-renormalized energy]
\label{lem:no-diffuse-limit}
The harmonic form constructed above extends across $T$ and vanishes
identically on $X$.
\end{lemma}

\begin{proof}
Each $h_i$ is exact. If $\gamma\subset X\setminus T$ is a smooth closed
curve, then smooth local convergence gives
\[
    \int_\gamma h
    =
    \lim_{i\to\infty}\int_\gamma h_i
    =0.
\]
Thus the closed form $h$ is globally exact on $X\setminus T$. Its $L^2$
bound gives removable singularities across the finitely many punctures.
For completeness, let \(\chi_\delta\) vanish on \(B_\delta(x_a)\), equal
one outside \(B_{2\delta}(x_a)\), and satisfy
\(|d\chi_\delta|\leqslant C\delta^{-1}\).  In dimension three,
\(\|d\chi_\delta\|_{L^2}=O(\delta^{1/2})\).  Testing the equations
\(dh=d^*h=0\) against \(\chi_\delta\) times a smooth
compactly supported form, and then letting \(\delta\downarrow0\), shows that they hold
distributionally across \(x_a\).  Interior elliptic regularity for the
Hodge Laplacian makes the extension smooth.  Since every loop in a
three-manifold can be perturbed away from finitely many points, the
extended form still has zero periods.  Thus
$h=df$ for a harmonic function $f$ on $X$ with
$df\in L^2$. The
finite Dirichlet energy Liouville theorem
\cite{fadel2023asymptotics}*{Theorem~2.22} implies that $f$ is constant.
Hence $h=0$.
\end{proof}

Lemma~\ref{lem:no-diffuse-limit} and
\eqref{eq:diffuse-part-measure} now give
\begin{equation}\label{eq:concentration-measure-prequantized}
    \mu_i\wto\sum_{a=1}^{\ell}\Theta(x_a)\delta_{x_a}
    \qquad\text{as Radon measures on }X,
\end{equation}
where $T=\{x_1,\ldots,x_\ell\}$ and
\[
    \Theta(x_a):=\mu(\{x_a\})\geqslant\eta_0
\]
is the concentration weight at $x_a$.

\begin{proposition}[Quantization and Higgs zeros]
\label{prop:quantization-zeros}
For every $x\in T$ there is a positive integer $K_x$ such that
\[
    \Theta(x)=4\pi K_x.
\]
Moreover, $T=S=Z$.
Consequently,
\begin{equation}\label{eq:quantized-measure-limit}
    \mu_i\wto4\pi\sum_{a=1}^{\ell}K_{x_a}\delta_{x_a},
    \qquad
    \sum_{a=1}^{\ell}K_{x_a}\leqslant k.
\end{equation}
\end{proposition}

\begin{proof}
Fix \(x\in T\).  Choose \(r>0\) so that
\(\overline{B_{2r}(x)}\cap T=\{x\}\) and
\(\mu(\partial B_r(x))=0\).  By
Lemmas~\ref{lem:small-energy-limit} and~\ref{lem:no-diffuse-limit}, the
fields are zero-free on a neighbourhood of \(\partial B_r(x)\) for all
large \(i\), and, writing \(\Psi_i=\phi_i/|\phi_i|\), one has
\[
    |\phi_i|\longrightarrow1,
    \qquad
    \text{d}_{A_i}\Psi_i\longrightarrow0
\]
smoothly there.  We shall also use the corresponding rate.  On this fixed
neighbourhood, the Green representation and the smooth convergence of
\(m_i^{-1}e_i\) imply
\[
    1-|\phi_i|=O(m_i^{-1}),
    \qquad
    m_i^{-1/2}|F_{A_i}|=O(1).
\]
Indeed,
\[
    m_i\bigl(1-|\phi_i|^2\bigr)
    =
    2\int_XG(\,\cdot\,,y)\,d\mu_i(y),
\]
and the potentials on the right are uniformly bounded on compact subsets
of \(X\setminus T\): locally this follows from the smooth density
convergence, while the remaining contribution is controlled by the total
mass and the Green estimate.  Let
\[
    K_i:=\deg\bigl(\Psi_i|_{\partial B_r(x)}\bigr)\in\mathbb Z.
\]
The Bogomolny equation and the Bianchi identity give the exact flux
identity
\[
    \mu_i(B_r(x))
    =\int_{\partial B_r(x)}\langle\phi_i,F_{A_i}\rangle.
\]
With the inner product and orientation conventions of the published
paper, the Chern--Weil identity on \(\partial B_r(x)\) is
\[
 \int_{\partial B_r(x)}\langle\Psi_i,F_{A_i}\rangle
 =4\pi K_i
 +\frac12\int_{\partial B_r(x)}
   \langle\Psi_i,[\text{d}_{A_i}\Psi_i,\text{d}_{A_i}\Psi_i]\rangle.
\]
The second integral tends to zero.  The preceding rate estimates also
give
\[
 \left|
 \int_{\partial B_r(x)}
 \langle\phi_i-\Psi_i,F_{A_i}\rangle
 \right|
 \leqslant
 C\,m_i^{-1/2}
 \longrightarrow0.
\]
Thus replacing \(\phi_i\) by \(\Psi_i\) in the flux produces an
\(o(1)\) error.  Since
\(\mu_i(B_r(x))\to\Theta(x)\), it follows that
\[
    \Theta(x)=4\pi K_i+o(1).
\]
The integers \(K_i\) are therefore eventually constant.  Denoting their
common value by \(K_x\), positivity of \(\Theta(x)\) gives
\(K_x\in\mathbb N\) and
\[
    \Theta(x)=4\pi K_x.
\]

Every finite accumulation point of Higgs zeros belongs to \(T\) by the
small energy lower bound.  Conversely, the nonzero degree \(K_x\) on
\(\partial B_r(x)\) forces a zero of \(\Phi_i\) in \(B_r(x)\) for all
large \(i\); letting \(r\downarrow0\) gives \(x\in Z\).  Thus
\(T=Z\).

It remains to identify \(T\) with \(S\).  If \(x\in T\), then every ball
about \(x\) carries limiting mass at least \(\Theta(x)>0\), so
\(x\in S\).  If \(x\notin T\), choose \(r>0\) with
\(\overline{B_r(x)}\cap T=\varnothing\); since the limiting measure is
supported on \(T\), one has \(\mu_i(B_r(x))\to0\), and hence
\(x\notin S\).  Therefore \(S=T=Z\).  Finally,
\(\mu_i(X)=4\pi k\) and vague convergence give
\(\sum_{x\in S}K_x\leqslant k\).
\end{proof}

Having established $T=S=Z$, we henceforth use the published notation $S$.
We next explain why the equation selects the scale $m_i^{-1}$ and why the
resulting Euclidean profiles do not bubble again.

For points $p_i\to x$ and scales $r_i\downarrow0$, define
\[
    s_{i,p_i}^{(r_i)}(z):=\exp_{p_i}(r_i z),
    \qquad
    g_{i,p_i}^{(r_i)}:=r_i^{-2}(s_{i,p_i}^{(r_i)})^*g,
\]
and rescale the Higgs field by its natural Bogomolny weight,
\[
    \widehat\Phi_i^{(r_i)}
    :=r_i(s_{i,p_i}^{(r_i)})^*\Phi_i.
\]
Then
\begin{equation}\label{eq:general-monopole-rescaling}
    *_{g_{i,p_i}^{(r_i)}}F_{(s_{i,p_i}^{(r_i)})^*A_i}
    =
    d_{(s_{i,p_i}^{(r_i)})^*A_i}
    \widehat\Phi_i^{(r_i)},
\end{equation}
and the rescaled mass is $r_im_i$. Moreover, for every fixed $R>0$,
\begin{equation}\label{eq:energy-rescaling-general-scale}
\begin{split}
&\mathscr E_{B_R(0),g_{i,p_i}^{(r_i)}}
 \left((s_{i,p_i}^{(r_i)})^*A_i,
       \widehat\Phi_i^{(r_i)}\right) \\
&\qquad=
 r_i\,\mathscr E_{B_{Rr_i}(p_i),g}(A_i,\Phi_i)
 =
 (r_im_i)\,\mu_i(B_{Rr_i}(p_i)).
\end{split}
\end{equation}

\begin{lemma}[Selection of the mass scale and absence of secondary bubbling]
\label{lem:mass-scale selection}
Suppose that a rescaling as above converges smoothly locally, modulo gauge,
to a nontrivial finite energy Euclidean monopole of finite mass. Then,
after passing to a subsequence,
\[
    r_im_i\longrightarrow\lambda\in(0,\infty).
\]
Thus every nontrivial finite mass Euclidean profile occurs at a scale
comparable to $m_i^{-1}$; after applying the standard Euclidean
Bogomolny scaling to the coordinates and the Higgs field, it may be
rescaled to have mass one.

At the distinguished scale $r_i=\eps_i=m_i^{-1}$, pointed sequences with
uniformly bounded charge are smoothly locally precompact modulo gauge and
cannot develop secondary bubbles at scales $o(\eps_i)$. Any nontrivial
pointed limit obtained by centring at a zero of $\Phi_i$ is a mass one
Euclidean monopole.
\end{lemma}

\begin{proof}
The first assertion follows directly from
\eqref{eq:general-monopole-rescaling}. If $r_im_i\to0$, then the total
energy of the rescaled configuration is bounded by
\[
    r_i\mathscr E_X(A_i,\Phi_i)
    =4\pi k\,r_im_i\longrightarrow0
\]
by the monopole energy identity. Hence every smooth finite energy limit is
flat with a covariantly constant Higgs field of mass zero, and is therefore
trivial. Suppose instead that $r_im_i\to\infty$. Fix
\(\delta\in(0,1)\).  The estimate for the radius of the region where the Higgs field is
small, either Theorem~\ref{thm: taubes_estimate} or its finite energy
version \eqref{eq:finite-energy-radius-estimate}, applied with centre
$p_i$, gives
a point \(y_i\) with
\[
 d(p_i,y_i)\leqslant R_\delta m_i^{-1},
 \qquad |\Phi_i(y_i)|\geq\delta m_i,
\]
where \(R_\delta\) is independent of \(i\).  In the \(r_i\)-rescaled
coordinates this point has distance at most
\(R_\delta/(r_im_i)\to0\) from the origin, whereas the rescaled Higgs
field has norm at least \(\delta r_im_i\to\infty\).  This contradicts
smooth local convergence to a finite valued Euclidean configuration. A
nontrivial finite mass profile consequently requires
$r_im_i\to\lambda\in(0,\infty)$.  Applying the standard Euclidean
Bogomolny scaling with factor $\lambda$ converts the limiting mass to
one.

Now take $r_i=\eps_i$. The rescaled configurations have mass one and total
energy at most $4\pi k$ on their expanding domains.
Lemma~\ref{lem:general-YMH-coarse-density}, Uhlenbeck gauge fixing, and
elliptic bootstrapping, exactly as in
Lemma~\ref{lem:zero-centred-mass-scale}, give smooth local precompactness.
Suppose that a pointed
sequence developed a further concentration at Euclidean scales
$\rho_i\downarrow0$. Rescaling a second time by $\rho_i$ multiplies the
energy by $\rho_i$; hence its total energy is at most
$4\pi k\rho_i\to0$. On the other hand, choosing the second scale by the
usual concentration function or maximal energy density normalization and
using the corrected $\varepsilon$-regularity estimate would produce a
nontrivial limit with a fixed positive amount of energy on a unit ball.
This contradiction excludes secondary concentration. This is only
a pointed compactness statement: global noncompactness remains possible
because different monopole constituents may separate to spatial infinity.

It remains to identify the mass of a nontrivial pointed limit. Write the rescaled mass-normalized fields as
\[
    \widetilde\phi_i=(s_{i,p_i}^{(\eps_i)})^*\phi_i.
\]
The Green representation for $m_i^2-|\Phi_i|^2$ from
\cite{fadel2023asymptotics}*{Theorem~3.11}, divided by $m_i^2$, reads
\begin{equation}\label{eq:rescaled-green-representation}
    1-|\widetilde\phi_i(z)|^2
    =2\int_X
      \eps_iG\bigl(s_{i,p_i}^{(\eps_i)}(z),y\bigr)\,d\mu_i(y).
\end{equation}
On a fixed rescaled ball, the kernels
\[
    \eps_iG\bigl(s_{i,p_i}^{(\eps_i)}(z),
                 s_{i,p_i}^{(\eps_i)}(w)\bigr)
\]
converge locally away from the diagonal to the Euclidean Green kernel
$G_{\mathbb R^3}(z,w)$. The total masses of the measures $\mu_i$ are
bounded by $4\pi k$. For $z$ in a fixed compact set, the contribution of
the region $|w|>L$ in the rescaled chart is bounded by $Ck/L$, while the
contribution from outside a fixed geodesic ball about $x$ is $O(\eps_i)$.
Thus sources which escape to infinity in the rescaled coordinates
contribute zero on compact sets. Passing first to $i\to\infty$ and then to
$L\to\infty$ in \eqref{eq:rescaled-green-representation} yields the
Euclidean Green representation with constant term one,
\[
    1-|\Phi_\infty(z)|^2
    =2\int_{\mathbb R^3}G_{\mathbb R^3}(z,w)
      |\text{d}_{A_\infty}\Phi_\infty(w)|^2\,dw.
\]
The Newtonian potential on the right tends to zero as $|z|\to\infty$ by
\cite{fadel2023asymptotics}*{Theorem~3.11}.
Thus $|\Phi_\infty|\to1$, and the limit has mass one. If the centres are
zeros, then $\Phi_\infty(0)=0$, so the limit is nontrivial.
\end{proof}

The preceding lemma shows that the only loss of compactness at a point
$x\in S$ is translation of mass one monopoles to infinity in the
coordinates rescaled by the mass. Profiles that remain at bounded mutual distance are contained
in one, possibly higher charge, Euclidean monopole; distinct Euclidean
profiles must separate by distances much larger than $m_i^{-1}$.

\begin{proposition}[Complete moving centre bubble cluster]
\label{prop:cluster-identity}
For every $x\in S$, there exist $N_x\geqslant1$, sequences of zeros
\[
    p_{i,x,\beta}\in\Phi_i^{-1}(0),
    \qquad
    p_{i,x,\beta}\longrightarrow x,
    \qquad 1\leqslant\beta\leqslant N_x,
\]
and nontrivial mass one Euclidean monopoles
$(A_{x,\beta},\Phi_{x,\beta})$ of positive charges $q_{x,\beta}$ such that,
after gauge and passage to a subsequence,
\[
    s_{i,p_{i,x,\beta}}^*(A_i,\phi_i)
    \longrightarrow(A_{x,\beta},\Phi_{x,\beta})
    \qquad\text{smoothly locally on }\mathbb R^3,
\]
where $s_{i,p}(z)=\exp_p(\eps_i z)$. The centres of distinct profiles
satisfy
\begin{equation}\label{eq:mass-scale-separation}
    \frac{d(p_{i,x,\beta},p_{i,x,\gamma})}{\eps_i}
    \longrightarrow\infty
    \qquad(\beta\neq\gamma).
\end{equation}
Their charges exhaust the concentration multiplicity,
\begin{equation}\label{eq:cluster-charge-identity}
    K_x=\sum_{\beta=1}^{N_x}q_{x,\beta},
\end{equation}
and hence
\begin{equation}\label{eq:local-bubble-energy-identity}
    4\pi K_x
    =
    \sum_{\beta=1}^{N_x}
    \mathscr E_{\mathbb R^3}(A_{x,\beta},\Phi_{x,\beta}).
\end{equation}
More precisely, choose $r>0$ such that
$\overline{B_{2r}(x)}\cap S=\{x\}$ and
$\mu(\partial B_r(x))=0$. Then
\begin{equation}\label{eq:cluster-no-loss}
\lim_{R\to\infty}\limsup_{i\to\infty}
\mu_i\left(
B_r(x)\setminus
\bigcup_{\beta=1}^{N_x}B_{R\eps_i}(p_{i,x,\beta})
\right)=0.
\end{equation}
Thus, no mass-renormalized energy remains between the finite collection of mass one profiles and the ambient scale.
\end{proposition}

\begin{proof}
Fix $x\in S$ and a radius $r$ as in the statement. By
Proposition~\ref{prop:quantization-zeros}, for all sufficiently large $i$ the normalized Higgs field is nonzero on $\partial B_r(x)$ and has degree
$K_x$ there. Since $x\in Z$, choose zeros $p_{i,x,1}\to x$.
Lemma~\ref{lem:mass-scale selection}, applied at the mass scale, gives a
mass one Euclidean monopole $(A_{x,1},\Phi_{x,1})$ of charge
$q_{x,1}\geqslant1$.

Choose a large regular radius $R$ for this Euclidean monopole, so that its
Higgs field is nonzero on $\partial B_R(0)$ and has degree $q_{x,1}$ there.
Smooth convergence implies that, for all sufficiently large $i$, the
normalized Higgs field on
$\partial B_{R\eps_i}(p_{i,x,1})$ is nonzero and has the same degree. If
$q_{x,1}=K_x$, the extraction stops. Otherwise, degree additivity forces a
new zero outside this ball. Indeed, if the normalized Higgs field were
nonzero on
\[
    B_r(x)\setminus\overline{B_{R\eps_i}(p_{i,x,1})},
\]
the sum of the degrees on its oriented boundary components would vanish,
which would give $K_x=q_{x,1}$, a contradiction.

Choose such a zero as a second centre. Taking first $i\to\infty$ and then
letting $R\to\infty$ by a diagonal argument produces a sequence
$p_{i,x,2}\to x$ satisfying
$d(p_{i,x,1},p_{i,x,2})/\eps_i\to\infty$.  To see the first convergence
without any hidden compactness assumption, choose \(r\) above so that
\(\overline{B_{2r}(x)}\cap Z=\{x\}\), which is possible because
\(Z=S\) is finite.  Every accumulation point of the selected zeros in
\(\overline{B_r(x)}\) must then be \(x\).
Lemma~\ref{lem:mass-scale selection} gives a second mass one Euclidean
monopole of positive charge $q_{x,2}$. Repeating the argument after deleting
large, pairwise disjoint mass scale balls about the already selected centres
produces profiles satisfying \eqref{eq:mass-scale-separation}. At every
finite stage, positivity of $\mu_i$ and smooth convergence on the selected
balls give
\[
    4\pi\sum_{\beta=1}^{j}q_{x,\beta}
    \leqslant 4\pi K_x,
\]
so the accumulated charge can never exceed $K_x$. If the inequality is
strict, degree additivity again forces a zero in the complement and hence
another profile. Since every $q_{x,\beta}$ is a positive integer, the
procedure terminates after at most $K_x$ steps, and at termination
\[
    \sum_{\beta=1}^{N_x}q_{x,\beta}=K_x.
\]
This proves \eqref{eq:cluster-charge-identity}.

It remains to verify that no mass-renormalized energy is missed. For fixed $R$,
the balls $B_{R\eps_i}(p_{i,x,\beta})$ are pairwise disjoint for all large
$i$ by \eqref{eq:mass-scale-separation}. Positivity and additivity of the
measure give
\[
\begin{split}
&\mu_i\left(
B_r(x)\setminus
\bigcup_{\beta=1}^{N_x}B_{R\eps_i}(p_{i,x,\beta})
\right)\\
&\qquad=
\mu_i(B_r(x))
-
\sum_{\beta=1}^{N_x}
\mu_i(B_{R\eps_i}(p_{i,x,\beta})).
\end{split}
\]
The first term tends to $4\pi K_x$ because
$\mu(\partial B_r(x))=0$. For each $\beta$, smooth pointed convergence and
the scaling identity give
\[
    \lim_{i\to\infty}
    \mu_i(B_{R\eps_i}(p_{i,x,\beta}))
    =
    \mathscr E_{B_R(0)}(A_{x,\beta},\Phi_{x,\beta}).
\]
Taking $i\to\infty$ and then $R\to\infty$, and using
\[
    \mathscr E_{\mathbb R^3}(A_{x,\beta},\Phi_{x,\beta})
    =4\pi q_{x,\beta},
\]
together with \eqref{eq:cluster-charge-identity}, gives
\eqref{eq:cluster-no-loss} and \eqref{eq:local-bubble-energy-identity}.
\end{proof}

\begin{remark}[Canonicity of the cluster data]
\label{rem:cluster-canonicity}
For a fixed subsequential limit of the mass-renormalized energy measures, the
integer \(K_x\) and the concentration weight \(4\pi K_x\) are intrinsic.
The individual profile decomposition need not be unique: the number
\(N_x\), the charges \(q_{x,\beta}\), and the limiting Euclidean
monopoles may depend on a further subsequence and on the choice of moving
centres.  The word ``complete'' refers to the energy exhaustion identity
\eqref{eq:cluster-no-loss}, not to uniqueness of the list of profiles.
\end{remark}

\begin{remark}[Concentration weight and individual bubbles]
\label{rmk:cluster-versus-bubble-charge}
For each $x\in S$, the integer $K_x$ records the total charge concentrated
at $x$; equivalently, $\Theta(x)=4\pi K_x$ is its concentration weight in
the limiting energy measure. In the published part~\textup{(a)}, the charge $k_x$ may be taken to be
the charge of any one selected bubble and therefore satisfies
$k_x\leqslant K_x$. The coefficient asserted in the published
part~\textup{(d)} need not equal the charge of this selected bubble.
Proposition~\ref{prop:hierarchical-cluster-example} gives an explicit
charge four family for which the selected fixed centre bubble has charge
one while $K_0=4$.  The missing profiles are obtained by translating the
same sequence rescaled by the mass to the other profiles; they do not arise by
rescaling again at smaller scales. The invariant statements are the
cluster charge identity \eqref{eq:cluster-charge-identity} and the no loss
identity \eqref{eq:cluster-no-loss}.
\end{remark}

\begin{remark}[Energy escaping through the end]
\label{rem:energy-escape-end}
Since \(\mu_i(X)=4\pi k\), the family \((\mu_i)\) is tight if and only if
\[
    \sum_{x\in S}K_x=k.
\]
Thus \(4\pi\bigl(k-\sum_{x\in S}K_x\bigr)\) is precisely the
mass-renormalized energy which escapes through the asymptotically
conical end rather than
concentrating at a finite point of \(X\).
\end{remark}

\begin{remark}[Corrections and additional results in version~5]
The complete cluster decomposition and the hierarchical example are new
strengthenings. The correction to the published measure statement consists
of restoring the factor $1/2$ in the displayed energy density and replacing
the charge $k_x$ of one selected bubble by the total cluster charge $K_x$
in part~\textup{(d)} of Theorem~\ref{thm: Main_Monopoles}. The
supporting corrections are the moving centre arguments, the rank one
restriction in Theorem~\ref{thm: Main_YMH}\textup{(a')}, the
uniform-threshold formulation of its part~\textup{(b')}, and the
regularity proof uniform in the mass.
\end{remark}

\section{Convergence as measures}\label{sec: measures}

\noindent\textbf{Corrected and expanded in version~5.}
The local quantization and complete cluster identity established in
Section~\ref{sec: bubbling} now give the corrected measure statement.
The coefficient at a concentration point is the total cluster charge
$K_x$, rather than the charge of one selected bubble.
\medskip

\begin{proof}[Proof of Theorem~\ref{thm: Main_Monopoles}]
Proposition~\ref{prop:quantization-zeros} gives the quantized convergence
of the energy measures, the equality $S=Z$, and
$\sum_{x\in S}K_x\leqslant k$. Proposition~\ref{prop:cluster-identity}
gives the complete moving centre cluster, the local energy and charge
identities, and the no loss statement. It also gives
$1\leqslant N_x\leqslant K_x$ and
\[
    \sum_{x\in S}N_x
    \leqslant
    \sum_{x\in S}K_x
    \leqslant k.
\]
The quantized characterization of $S$ in part~\textup{(b)} follows from
Proposition~\ref{prop:quantization-zeros}. Finally, if
$S\neq\varnothing$, then
\[
    \mathcal H^0(S)\min_{x\in S}K_x
    \leqslant
    \sum_{x\in S}K_x
    \leqslant k,
\]
which proves part~\textup{(c)}. The tightness criterion follows from
$\mu_i(X)=4\pi k$ and the local identity
$\mu(\{x\})=4\pi K_x$.
\end{proof}

\section{\texorpdfstring{The proof of assertion
\eqref{eq:Zeroes_Phi}}{The proof of the zero-set localization assertion}}
\label{appendix: A}

In this section, we shall prove assertion \eqref{eq:Zeroes_Phi}, which says that the zeros of the monopoles constructed via Theorem \ref{thm: Monopoles_Examples} are contained in balls of radius $10m^{-1/2}$ around the $k$ points in $X$ used in the construction. This requires a number of technical ingredients from \cite{oliveira2016monopoles} and so we decided to include this section as an Appendix.

It follows from \cite{oliveira2016monopoles}*{Proposition~6} that the monopole $(A_i, \Phi_i)$ can be written as $(A_i, \Phi_i) = (A^0_i, \Phi^0_i)+(a_i , \phi_i)$, where 

\begin{enumerate}

\item[A.a] $(A^0_i,\Phi_i^0)$ is an approximate monopole constructed in \cite{oliveira2016monopoles}*{Proposition~4}. Moreover, by its own construction, we have that the restriction 
\[
\Phi^0_i : X \backslash \cup_{j=1}^k B_{10 m_i^{-1/2}}(x_j) \rightarrow \mathfrak{su}(2),
\]
satisfies $\vert \Phi^0_i \vert \geq m_i/2$, is nowhere zero, and yields a splitting of the trivial rank-2 complex vector bundle $\underline{\mathbb{C}^2} \cong L \oplus L^{-1}$, where the complex line bundle $L$ is such that 
\[
\deg(L \vert_{\partial B_{10 m_i^{-1/2}}(x_j) })=1,
\]
for all $j =1, \ldots , k$. In other words, the restricted map 
\begin{equation}\nonumber
\Phi^0_i: \partial B_{10 m_i^{-1/2}}(x_j)  \rightarrow \mathfrak{su}(2) \backslash \lbrace 0 \rbrace ,
\end{equation} 
has degree $1$.

\item[A.b] $(a_i , \phi_i)\in \Gamma ((\Lambda^1 \oplus \Lambda^0) \otimes \mathfrak{su}(2))$ satisfies an elliptic equation, when in a certain Coulomb gauge (see \cite{oliveira2016monopoles}*{Lemma~13}). Moreover, from \cite{oliveira2016monopoles}*{Proposition~6}, it satisfies
\begin{equation}\label{eq:H_Estimate}
 \Vert (a_i , \phi_i) \Vert_{H_{1,-\frac{1}{2} } } \lesssim m_i^{-7/4},
\end{equation}
where $H_{1,-1/2}$ is a certain Sobolev space. 

\end{enumerate}

The Sobolev space $H_{1,\nu+1}$, with $\nu=-3/2$ here, is one of several $H_{n, \nu+n}$ constructed using the approximate monopole $(A^0_i,\Phi^0_i)$. These are well adapted to solving the monopole equation, and have the property that, in a certain gauge (see \cite{oliveira2016monopoles}*{Section~5}), one can iterate estimate \eqref{eq:H_Estimate} to obtain that 
\[
\Vert (a_i , \phi_i) \Vert_{H_{n,\nu+n} } \lesssim m_i^{-7/4},
\]
for all $n \in \mathbb{N}$. Moreover, once restricted to certain subsets of $X$, these spaces satisfy a number of interesting properties. Some of these can be easily read from the definition in \cite{oliveira2016monopoles}*{Section~4.1}, and we summarize them below.

\begin{enumerate}

\item[B.a] For every $n\in\mathbb N$ and every compact set $K\subset X$,
there is a constant $c_n(K)>0$, depending only on $g$, $K$, and $n$,
but not on $m_i$, such that
\begin{equation}
\Vert (a_i , \phi_i) \Vert_{L^{2,n} (K) } \leqslant c_n(K) \ m_i^2 \ \Vert (a_i , \phi_i) \Vert_{H_{n,\nu + n} (K) } \lesssim m_i^{1/4}  .
\end{equation}

\item[B.b] For $\epsilon>0$ we consider 
\[
C_{\epsilon} = X \backslash \cup_{j=1}^k B_{\epsilon}(x_j),
\]
i.e. the complement of the balls of radius $\epsilon$ centered at the points $x_j$.
If $k\geqslant2$, let $4d = \min_{j\neq l} \dist(x_j,x_l)$; if $k=1$, choose any fixed $d>0$. Then the balls of radius $d$ around the points $x_j$ are disjoint. Using $d$, we shall consider $C_d$. Certain weight functions $W_n$, on which the spaces $H_{n, \nu + n}$ depend, can be arranged so that 
\begin{equation}
\Vert (a_i , \phi_i) \Vert_{L^{2,n} (K) } \leqslant c_n(K) \ \Vert (a_i , \phi_i) \Vert_{H_{n,\nu + n} (K) } \lesssim m_i^{-7/4}  .
\end{equation}
for every compact set $K\subset C_d$.

\item[B.c] On a sufficiently far region $C$ of the AC end we can use the
fact that $\Phi^0_i$ is nonzero, as mentioned in A.a, to write any
$\mathfrak{su}(2)$-valued tensor $f$ as $f=f^{\Vert}+f^{\perp}$, with
$f^{\Vert}$ denoting the component parallel to $\Phi^0_i$ and $f^{\perp}$
the orthogonal one. On $C$, and for large $m_i$, we can write
\begin{equation}
\Vert (a_i^{\Vert} , \phi_i^{\Vert} ) \Vert_{L^{2,n}_{\nu + n}(C) } + \Vert (a_i^{\perp} , \phi_i^{\perp} ) \Vert_{L^{2,n}(C) } = \Vert (a_i , \phi_i) \Vert_{H_{n,\nu + n}(C)}   \lesssim m_i^{-7/4}  ,
\end{equation}
where the spaces $L^{2,n}_{\nu+n}$ are the more standard Lockhart-McOwen conically weighted spaces.

\end{enumerate}

Combining item~B.a with the Sobolev embedding
$L^{2,n}(K)\hookrightarrow C^{n-2}(K)$, for $n$ sufficiently large, we
obtain
\begin{equation}\label{eq:C^0_Estimate}
\begin{split}
    \norm{(a_i,\phi_i)}_{C^{n-2}(K)}
    &\lesssim m_i^{1/4},\\
    \norm{\Phi_i-\Phi_i^0}_{C^0(K)}
    &\lesssim m_i^{1/4}.
\end{split}
\end{equation}
for any compact set $K \subset X$. In particular, $(a_i,\phi_i)$ is smooth. Moreover, as mentioned in A.a, $\vert \Phi^0_i \vert \geq m_i/2$ in $C_{10m_i^{-1/2}}$ and thus
\[
\vert \Phi_i \vert \geq \vert \Phi^0_i \vert - \Vert \phi_i \Vert_{C^0} \geq \frac{m_i}{2} - c m_i^{1/4}, \ \ \text{in any compact  $K \subset C_{10m_i^{-1/2}}$}
\] 
and so, for $m_i \gg 1$, does not vanish in $\partial \overline{B_{10m_i^{-1/2}}(x_j)}$ for any $j \in \lbrace 1 , \ldots , k \rbrace$. On the fixed compact complement of $C$, the same argument applies with $K$ equal to that complement minus the indicated balls. Putting this together with the decay estimate in A.b on $C$, we conclude that any zero of $\Phi_i$ must lie inside one of the balls of radius $10m_i^{-1/2}$ around the points $x_j$. Furthermore, this estimate shows that the $1$-parameter family of maps 
\[
\Phi^t_i = \Phi^0_i + t \phi_i: C_{10m_i^{-1/2}} \rightarrow \mathfrak{su}(2) \backslash \lbrace 0 \rbrace,
\] 
gives a homotopy between $\Phi_i^0$ and $\Phi_i$. Combining this with the discussion in A.a we conclude that
\[
\deg( \Phi_i \vert_{\partial B_{10 m_i^{-1/2}} (x_j)} ) = \deg( \Phi^0_i \vert_{\partial B_{10 m_i^{-1/2}}(x_j)} ) =1.
\]
Thus, $\Phi_i$ has a zero inside $B_{10 m_i^{-1/2}}(x_j)$. 
%Moreover, by combining the fact that $\Phi^0_i \geq m_i/2$ outside $\cup_{i=1}^k B_{10 m_i^{-1/2}}(x_j)$ with the estimate \eqref{eq:C^0_Estimate} we conclude that, for large enough $i$, $\Phi_i$ has no zeroes in $X \backslash \cup_{i=1}^k B_{10 m_i^{-1/2}}(x_j)$.

%\input{appendix}

%===============================================================================
\bibliography{references}
%===============================================================================

%\input{glossario}

\end{document}